\documentclass[a4paper,11pt]{amsart}
\pdfoutput=1 
\usepackage{amsmath,amsthm,amssymb}
\usepackage[mathscr]{eucal}
 \usepackage{cite}
\usepackage{upgreek}
\usepackage[bookmarks=false]{hyperref}
\usepackage{enumerate}
\usepackage{amsmath}
\usepackage{mathrsfs}
\usepackage{blindtext}
\usepackage{scrextend}
\usepackage{enumitem}
\usepackage{bm} 
\addtokomafont{labelinglabel}{\sffamily}
\usepackage{color}
\usepackage[at]{easylist}
\usepackage{tikz}
\usepackage{svg}
\usepackage{graphicx}

\usepackage{comment}

\numberwithin{equation}{section}

\newtheorem{main}{Theorem}
\newtheorem{mprop}[main]{Proposition}

\newtheorem{thm}{Theorem}[section]
\newtheorem*{thm*}{Theorem}
\newtheorem{lem}[thm]{Lemma}
\newtheorem*{prob*}{Problem}

\newtheorem{prop}[thm]{Proposition}
\newtheorem*{prop*}{Proposition}

\newtheorem{cor}[thm]{Corollary}
\newtheorem*{cor*}{Corollary}

\theoremstyle{definition}
\newtheorem{defn}[thm]{Definition}
\newtheorem*{defn*}{Definition}

\newtheorem{remark}[thm]{Remark}

\newtheorem{claim}[thm]{Claim}
\newtheorem*{question*}{Question}
\newtheorem*{Pquestion*}{Popa's question}

\newtheorem*{conv*}{Convention}

\newcommand{\N}{\mathbb{N}}
\newcommand{\R}{\mathbb{R}}
\newcommand{\C}{\mathbb{C}}
\newcommand{\Z}{\mathbb{Z}}

\newcommand{\F}{\mathbb{F}}

\newcommand{\Q}{\mathbb{Q}}
\newcommand{\M}{\mathbb{M}}

\newcommand{\cF}{\mathcal{F}}

\newcommand{\cH}{\mathcal{H}}

\newcommand{\cN}{\mathcal{N}}

\newcommand{\cP}{\mathcal{P}}
\newcommand{\cQ}{\mathcal{Q}}

\newcommand{\cU}{\mathcal{U}}
\newcommand{\cV}{\mathcal{V}}

\newcommand{\cZ}{\mathcal{Z}}

\newcommand{\ee}{\varepsilon}

\newcommand{\Aut}{\operatorname{Aut}}

\newcommand{\eps}{\varepsilon}

\newcommand{\twoone}{II$_1$ }

\usepackage{tikz}
\usepackage{tikz-cd}

\begin{document}

\title[On conjugacy  and perturbations of subalgebras in II$_1$ factors]
{On conjugacy and perturbation of subalgebras}

\author[David Gao]{David Gao}
\address{Department of Mathematical Sciences, UCSD, 9500 Gilman Dr, La Jolla, CA 92092, USA}\email{weg002@ucsd.edu}\urladdr{https://sites.google.com/ucsd.edu/david-gao}

\author[Srivatsav Kunnawalkam Elayavalli]{Srivatsav Kunnawalkam Elayavalli}
\address{Department of Mathematical Sciences, UCSD, 9500 Gilman Dr, La Jolla, CA 92092, USA}\email{srivatsav.kunnawalkam.elayavalli@vanderbilt.edu}
\urladdr{https://sites.google.com/view/srivatsavke/home}

\author[Gregory Patchell]{Gregory Patchell}
\address{Department of Mathematical Sciences, UCSD, 9500 Gilman Dr, La Jolla, CA 92092, USA}\email{gpatchel@ucsd.edu}
\urladdr{https://sites.google.com/view/gpatchel/home}

\author[Hui Tan]{Hui Tan}
\address{Department of Mathematical Sciences, UCSD, 9500 Gilman Dr, La Jolla, CA 92092, USA}\email{hutan@ucsd.edu}\urladdr{https://tanhui3.wordpress.com/}

\begin{abstract}

We study conjugacy orbits of certain types of subalgebras in tracial von Neumann algebras. We construct a highly indecomposable non Gamma II$_1$ factor $N$ such that every separable von Neumann subalgebra of $N$ with Haagerup's property admits a unique embedding up to unitary conjugation. Such a factor necessarily has to be non separable, but we show that it can be taken of density character $2^{\aleph_0}$. On the other hand we are able to construct for any separable II$_1$ factor $M_0$, a separable II$_1$ factor $M$ containing $M_0$ such that every property (T) subfactor admits a unique embedding into $M$ up to uniformly approximate unitary equivalence; i.e., any pair of embeddings can be conjugated up to a small uniform $2$-norm perturbation.

\end{abstract}
\maketitle

\section{Statements of main results}

In this paper we make new contributions to the study of conjugacy and perturbation among  subalgebras of tracial von Neumann algebras. These topics have been of interest in the area for many decades, see for instance  \cite{christensen1977perturbation, christensen1979subalgebrasoffinite, PopaCorr, PopaL2Betti, PSSperturbations, popa2006strong, JungTubularity}.

\subsection*{Conjugating Haagerup subalgebras} Our first main result is the construction of the following exotic highly indecomposable non separable non Gamma II$_1$ factor:

\begin{main}\label{intro main theorem A}
   There exists a non Gamma II$_1$ factor $M$ such that   every separable subalgebra of $M$ with Haagerup's property embeds in $M$ uniquely up to unitary conjugacy. 
\end{main}

We first describe some key new ideas underlying our construction. Our II$_1$ factor is built via a concrete inductive construction consisting of iterated amalgamated free products, and is  tailored  to satisfy the conjugacy condition above. One of the main challenges one typically faces in this situation is to maintain non Gamma in the inductive limit. We adopt a strategy similar to \cite{CIKE23}: begin the construction with a diffuse property (T) von Neumann algebra and then carefully  apply the structure theorems of \cite{IPP} at each stage of  our construction, to get that the initial property (T) subalgebra is irreducible in the entire union. What allows us to carry this out is a new insight we make on the HNN extension construction in von Neumann algebra theory \cite{uedahnn, fimavaes}, which by design allows one to enlarge the algebra where one can conjugate isomorphic subalgebras.   We design an a priori new construction involving amalgamated free products with free wreath products to achieve the end goal of an HNN extension and prove that it is indeed isomorphic to the HNN extension, via a universality argument. It is known that the HNN extension  is isomorphic to a corner of an amalgamated free product (see Proposition 3.1 of \cite{ueda2008remarks}). The contribution we make here is to show that the HNN extension is a natural unital subalgebra of an amalgamated free product, allowing  us some  ease in applying the non intertwining criteria of \cite{IPP}. Upon arranging a  careful ordering of pairs of the subalgebras with Haagerup's property we inductively apply our HNN construction to conjugate them. Note that the passage to non separable is inevitable if  we are to  handle simultaneously all pairs of separable Haagerup subalgebras, as every abelian subalgebra has Haagerup's property. Moreover, our factor can be taken of density character $2^{\aleph_0}$. Note also that we can choose our $M$ to contain a copy of all separable Haagerup von Neumann algebras simply by beginning the iterative construction with the tensor product of all separable Haagerup von Neumann algebras. Another key ingredient in carrying out the entire procedure is a  flexibility involving bypassing taking  closures around certain limit ordinals (see also \cite{FH22}). 

 Our construction can also be naturally modified to yield (see Corollary \ref{EC cor}) a new approach to constructing existentially closed II$_1$ factor \cite{ecfactors}. This is a model theoretic notion that has been studied in the literature recently \cite{AGKE, goldbring2023uniformly, ecfactortensor,ioana2023existential}. Note that existentially closed factors possess property Gamma (they are moreover McDuff \cite{goldbring2013theory}), hence will necessarily not be isomorphic to the construction in Theorem A. As such  the Haagerup assumption cannot be removed from Theorem A,  for this reason.
 \begin{main}\label{main-B}
     There exists a II$_1$ factor $(N,\tau)$ of density character $2^{\aleph_0}$ that contains a unique copy of each separable tracial von Neumann algebra up to unitary conjugacy. Such a II$_1$ factor is necessarily existentially closed.
 \end{main}

\subsection*{Indecomposability and structural properties} We now describe the various indecomposability and structural properties of our factor $M$ from Theorem \ref{intro main theorem A}. These indicate that $M$ cannot  arise from more or less any standard von Neumann algebraic construction.  The following results here are recorded in Proposition \ref{ucc main prop}. 

First we have that $M\ncong \prod_{i\to \mathcal{U}}N_i$ where $N_i$ are finite factors and $\mathcal{U}$ is a countably incomplete ultrafilter on any infinite set. This follows from Corollary 3.8 of 
\cite{ScottSri2019ultraproduct}.  This result guarantees the existence inside any ultraproduct of II$_1$ factors of two embeddings of the free group factor $L(\mathbb{F}_2)$ that are not unitarily conjugate (in fact the result shows the same  for any Connes-embeddable non amenable II$_1$ factor). Since $L(\mathbb{F}_2)$ has Haagerup's property this contradicts the main feature of our construction.

 Results of \cite{popaortho} (see also \cite{kadisonsingerpopa, sorincourse})  prove   the following indecomposability properties for $M$ (in each case a pair  of Haar unitaries that are not conjugate are identified): $M$ does not admit a diffuse regular subalgebra $M_0$ such that there is a Haar unitary $u\in \mathcal{U}(M)$ satisfying $\{u\}''\perp M_0$. In particular $M\ncong N_1\mathbin{\bar{\otimes}} N_2$ where $N_i$ are II$_1$ factors; $M \ncong N\rtimes G$ where $N$ is a diffuse tracial von Neumann algebra and $G$ is an infinite group. Also, $M$ does not admit Cartan subalgebras. 
 
  We also obtain  using an elementary argument involving normal form decompositions, the following: $M\ncong N_1*_BN_2$ where $N_1$ is diffuse and there exists $u_i\in \mathcal{U}(N_i)$ for $i=1,2$ such that $\mathbb{E}_B(u_i)=0$; $M\ncong \text{HNN}(N, N_0, \theta)$ where $N$ is diffuse. Also, $M\ncong L(G)$ where  $G$ is a discrete group that admits two elements $g,h\in G$ such that  $g^n$ is not conjugate to $h^n$ for  all $n\in \mathbb{N}$.   
  
  As a final remark  one can also take  $h(M)\leq 0$ where $h$ denotes the 1-bounded entropy (\cite{Hayes2018, JungSB}). This result can be obtained by weaving into our inductive limit the 2-handle construction in \cite{CIKE23}, thereby ensuring sequential commutation which forces 1-bounded entropy to be non positive (see also \cite{patchellelayavalli2023sequential}). Moreover, by the arguments in \cite{patchellelayavalli2023sequential} our construction can be shown to admit a unique sequential commutation orbit internally in  the factor.

   \subsection*{Conjugating property (T) subalgebras} In the setting of working with property (T) subalgebras, we are able to achieve a result similar to Theorem \ref{intro main theorem A}, the significant difference being that the  resulting construction can be made separable as opposed to the setting of conjugating Haagerup subalgebras in which case the construction is forced to be non separable. However the natural notion of equivalence to consider here is what we refer to as \emph{uniformly approximate unitary equivalence}, as opposed to on the nose unitary equivalence, which is not likely possible to aim for in general in this setting in light of Remark \ref{remark crazy}.  We define this concept as: $\pi_i:N \to M$, $i=1,2$ are uniformly approximate unitary equivalent embeddings, if for all $\eps>0$ there exists a unitary $u\in \mathcal{U}(M)$ such that $\|\mathbb{E}_{\pi_2(N)}(u\pi_1(x)u^*)-u\pi_1(x)u^*\|_2<\eps$ for every $x\in (N)_1$). 

\begin{main}\label{main-(T)}
        For any separable II$_1$ factor $M_0$ there exists a separable II$_1$ factor $M$ containing $M_0$ such that any diffuse property (T) von Neumann subalgebra $N\subset M$ admits a unique embedding into $M$ up to uniform approximate unitarily equivalence 
\end{main}

We actually prove a stronger statement than the above (see Theorem \ref{stronger statement}) which also further demonstrates the optimality of the result. The proof of Theorem \ref{main-(T)} uses a previously-studied metric between subalgebras (see \cite{christensen1977perturbation,christensen1977perturbationsII,christensen1979subalgebrasoffinite,PSSperturbations,WangPerturbations}); we denote this metric $d$. When two subalgebras are close with respect to $d,$ they have large corners which are isomorphic \cite{PSSperturbations}; similarly, if one subalgebra is almost included in another then there exists a genuine inclusion on large corners \cite{christensen1979subalgebrasoffinite}. We also argue as in \cite{PopaCorr} that the set of property (T) subalgebras of a separable \twoone factor is separable with respect to $d$ (see also \cite{ConnesCountable,anantharaman-popa}). These notions are made precise and estimates are given in Section \ref{sect-pf-of-(T)-conj}. 

We now describe the construction of $M$ in Theorem \ref{main-(T)}.  We proceed by induction as in Theorem \ref{intro main theorem A} (see also \cite{CIKE23}) to construct $M$ from $M_0.$ Given \twoone factors $M_0\subset M_1\subset\ldots\subset M_n,$ we pick a dense set $D_n$ of property (T) subalgebras and recursively apply our HNN extension to get that every pair of subalgebras $B_1,B_2 \in D_n$ which are isomorphic on large corners are nearly unitarily conjugate in $M_k$ for some $k\geq n$. We define $M$ to be the inductive limit of the $M_n.$ Given two isomorphic property (T) subalgebras $N_i$, $i=1,2$ of $M,$ we use property (T) and \cite{christensen1979subalgebrasoffinite} to embed large corners of each $N_i$ into some $M_n$. Each $N_i$ is $d$-close to some $B_i$ in $D_n$, forcing $B_1,B_2$ to be isomorphic on large corners. By construction, the $B_i$ are nearly unitarily equivalent in $M$ so that the $N_i$ are also nearly unitarily equivalent.

    One consequence of uniformly approximate unitary equivalence is bi-intertwining (\cite{PopaCorr, popa2006strong}): if $\pi_i:N \to M$, $i=1,2$ are two uniformly approximately  unitarily equivalent embeddings, then it follows that $\pi_1(N)\prec_M\pi_2(N)$ and $\pi_2(N)\prec_M\pi_1(N)$. This follows from results of \cite{christensen1979subalgebrasoffinite}, see also \cite{PopaCorr, PSSperturbations}. This puts certain non trivial  restrictions on the structure of $M$; for instance, this implies that our factor $M$ from above cannot be isomorphic to $M_1*M_2$ where each $M_i$ contains a copy of the same diffuse property (T) subalgebra. These two embeddings of of this subalgebra cannot bi-intertwine, contradicting this feature of our construction for $M$.
\subsection*{Two results of independent interest}
We conclude  the paper with two results of independent interest which arose from related considerations. The first is a genericity statement about elements generating factors inside tracial von Neumann algebras, see Proposition \ref{factor-lifting}. 

\begin{mprop}
        Let $\{x_1, \cdots, x_n\} \subset (M)_1$ be a finite set of elements of operator norm at most $1$ in a \twoone factor $M$. Then for any $\eps > 0$, there exists $\{y_1, \cdots, y_n\} \subset (M)_1$, elements of operator norm at most $1$, such that $\{y_1, \cdots, y_n\}$ generates a subfactor of $M$ and furthermore $\|y_i - x_i\|_2 < \eps$ for all $i$.
\end{mprop}

A delicate perturbation theorem in the above flavor has been obtained recently in \cite{ioana2024tracespacesfreeproduct}, with applications to certain problems on trace spaces.     

The second result is a sentence  that perhaps could be used  effectively to distinguish between non Gamma factors up to elementary equivalence, see Proposition \ref{gamma-prop}.    

\textbf{Acknowledgements} We thank Adrian Ioana, David Jekel and Jesse Peterson for several helpful comments and suggestions. We thank David Jekel and Ben Hayes for suggesting the proof of Proposition \ref{factor-lifting}. We would like to thank Brent Nelson without whose timely support this paper would not have materialized.  
\section{Preliminaries}

\subsection{Notation and Background for \twoone Factors}

In a tracial von Neumann algebra $(M,\tau),$ $\|\cdot\|_2$ is the norm defined by $\|x\|_2 = \sqrt{\tau(x^*x)}$. All inclusions of von Neumann algebras are unital unless otherwise specified. If $N\subset (M,\tau)$ is a von Neumann subalgebra, then there is a unique normal trace-preserving conditional expectation $E_N:M\to N.$  

If $(M_\lambda)_{\lambda < \kappa} \subset N$ is an increasing chain of tracial von Neumann subalgebras with compatible traces of a von Neumann algebra $N$, then $M := \overline{\cup_{\lambda<\kappa} M_\lambda}^{SOT}$ is a tracial von Neumann subalgebra of $N$ called the \emph{inductive limit} of $M_\lambda.$  The following is a basic fact (see, for example, Lemma 2.3 of \cite{patchell2023primess}). 

\begin{lem}
    Let $(M,\tau) = \overline{\cup_{\lambda < \kappa} M_\lambda}^{SOT}$ be a tracial von Neumann algebra realized as an inductive limit where $\kappa$ is an ordinal. Then $M_1'\cap M = \overline{\cup_{\lambda < \kappa} M_1'\cap M_\lambda}^{SOT}$.
\end{lem}



Two von Neumann subalgebras $N_1,N_2 \subset (M,\tau)$ are said to be \emph{orthogonal} if $\tau(n_1n_2) = \tau(n_1)\tau(n_2)$ for any choice of $n_i\in N_i$. In this case, we write $N_1 \perp N_2.$ We say an element $x$ is \emph{orthogonal} to a subalgebra $N\subset (M,\tau)$ if $\tau(nx) = 0$ for all $n\in N.$

\begin{lem}[Corollary 2.6 of \cite{popaortho}]\label{popa-ortho-cor2.6}
    Let $B\subset (M,\tau)$ be a von Neumann subalgebra and $u\in \cU(M)$ a unitary. If there is $B_0\subset B$ a diffuse von Neumann subalgebra such that $uB_0u^* \perp B,$ then $u$ is orthogonal to $\cN(B)''.$
\end{lem}


    

    

In this paragraph we follow Section 3 of \cite{Popa93}. Let $(M_i,\tau_i)$ be tracial von Neumann algebras, $i=1,2$ with a common von Neumann subalgebra $B.$ Let $E_i:M_i\to B$ denote the trace-preserving conditional expectation. The \emph{amalgamated free product} $M_1 *_B M_2$ is a tracial von Neumann algebra which is SOT-densely spanned by elements of the form $x=b\in B$ and $x = x_{1,i_1}\cdots x_{k,i_k}$ where $x_{j,i_j}\in M_{i_j}$, $E_{i_j}(x_{j,i_j}) = 0$ (in other words, $x_{j,i_j}\in M_i\ominus B$), and $i_1\neq \ldots \neq i_k$. We call such elements reduced words. If $x\not\in B$ is a reduced word, then $E_B(x) = 0$, and in particular, $\tau(x)=0.$ We note that the subspaces $B$ and $(M_{i_1}\ominus B)\cdots(M_{i_k}\ominus B)$ are all orthogonal for different tuples $(i_1,\ldots,i_k)$ such that $i_1\neq \ldots \neq i_k$. In the case $B = \C1$ we write $M_1 * M_2$ and call it the free product of $M_1$ with $M_2.$ 

For a sequence $(M_n,\tau_n)_{n\in \Z}$ of tracial von Neumann algebras with common von Neumann subalgebra $B,$ we recall that the amalgamated free product is associative, so that $(M_1 *_B M_2) *_B M_3 \cong M_1 *_B (M_2 *_B M_3)$, so we write simply $M_1*_B M_2*_B M_3.$ We define $*^{n\in \Z}_B M_n$ to be the inductive limit of the algebras $M_{-N} *_B M_{-N+1} *_B \cdots *_B M_N$ as $N\to\infty.$ If all of the $M_n,$ $n\in\Z$ are isomorphic to $(M,\tau)$, then there is a natural permutation automorphism $\sigma \in \Aut (*_B^{n\in\Z} M_n)$ which fixes $B$ and takes $M_n$ to $M_{n+1}$. The group this automorphism generates is isomorphic to $\Z$, and we define the \emph{amalgamated free wreath product} $M \wr_B^* \Z$ as equal to the crossed product $*_B^{n\in\Z} M_n \rtimes_\sigma \Z$. We denote $M\wr^* \Z = M \wr^*_\C \Z$. We note too that there is a natural isomorphism $M \wr_B^* \Z \cong M *_B (B\overline{\otimes}L\Z)$ (see, e.g., Section 3 of \cite{CartanAFP}).

\begin{defn}
    A subalgebra $B$ of tracial von Neumann algebra $(M,\tau)$ is said to be relatively rigid (or the inclusion has relative Property (T)) if either of the following conditions hold:
    \begin{enumerate}
        \item For all $\ee>0,$ there is $F\subset M$ finite and $\delta>0$ such that for all unital, tracial, completely positive maps $\phi:M\to M$ satisfying $\max_{x\in F}\|\phi(x)-x\|_2<\delta,$ we have $\|\phi(b)-b\|_2< \ee$ for all $b\in (B)_1$ (the operator norm unit ball of $B$).
        \item For all $\ee>0$ there is $F\subset M$ finite and $\delta>0$ such that for all $M$-$M$ bimodules $\cH$ and tracial vectors $\xi \in \cH$ satisfying $\max_{x\in F}\|x\xi - \xi x\| < \delta,$ there is $\eta\in\cH$ such that $\eta$ is $B$-central (i.e., $\eta b = b\eta$ for all $b\in B$) and $\|\xi-\eta\| < \ee.$
    \end{enumerate}
    We say $M$ is rigid (or has Property (T)) if $M\subset M$ is relatively rigid.
\end{defn}

We record a minor technical result which be used in the proof of Theorem \ref{main-(T)}. The proof is nearly identical to the proof of Proposition 14.2.4 in \cite{anantharaman-popa}.

\begin{lem}\label{stupid-T-lemma}
    Suppose $(M,\tau)$ is a tracial von Neumann algebra and $B\subset M$ is a subalgebra with Property (T). Then for every $\ee>0,$ there exist a finite set $F\subset B$ and a $\delta > 0$ such that whenever $\phi:M \to M$ is a tracial, unital, completely positive map such that $\max_{x\in F}\|\phi(x)-x\|_2 < \delta,$ then $\|\phi(b)-b\|_2 < \ee$ for all $b\in (B)_1$.
\end{lem}

\begin{proof}
    Let $\ee>0.$ Set $\ee' = \ee/2$. Since $B$ has Property (T), we can find $F\subset B$ finite and $\delta>0$ such that for all $B$-$B$ bimodules $\cH$ and tracial vectors $\xi\in\cH$ such that $\max_{x\in F}\|x\xi - \xi x\| < \delta,$ there is a $B$-central vector $\eta$ such that $\|\eta-\xi\| < \ee'.$ 

    Now let $\phi:M\to M$ be a tracial, unital, completely positive map such that $\max_{x\in F}\|\phi(x)-x\|-2 < \delta'$, where $2\delta'\max_{x\in F}\|x\|_2 = \delta^2$. Let $(\cH,\xi)$ be the pointed $M$-$M$ bimodule associated to $\phi$. Then $\|x\xi - \xi x\|^2 \leq 2\|\phi(x)-x\|_2\|x\|_2 < \delta^2$. Since $\cH$ is an $M$-$M$ bimodule, it is in particular a $B$-$B$ bimodule, and so there is a $B$-central vector $\eta \in \cH$ such that $\|\eta-\xi\| < \ee'.$ Then for $b\in (B)_1,$ we have $\|\phi(b)-b\|_2 \leq \|b\xi-\xi b\| \leq \|b(\xi-\eta) - (\xi-\eta)b\| \leq 2\ee' = \ee.$
\end{proof}

We omit the definitions of Haagerup's property but it can be found in Chapter 16 of \cite{anantharaman-popa}. We record the following facts about Haagerup's property and Property (T).

A group $\Gamma$ has Property (T) if and only if $L\Gamma$ has Property (T). In particular, $L(SL_3(\Z))$ has Property (T). Furthermore, if $M$ has property (T) then so does $pMp$ for any projection $p\in M$ (Chapter 14 of \cite{anantharaman-popa}). Furthermore, a \twoone factor $M$ has Property (T) if and only if whenever $N$ is a tracial von Neumann algebra such that $M\subset N,$ we have $M'\cap N^\cU = (M'\cap N)^\cU$ for a free ultrafilter $\cU$ on $\N$ \cite{tan2023spectral}. The free group factors $L(\F_n)$ all have Haagerup's property (Chapter 16 of \cite{anantharaman-popa}). If $M$ has Haagerup's property then so does $pMp$, and a free product of \twoone factors with Haagerup's property again has Haagerup's property (Corollary 5 of \cite{Boca_1993}). A diffuse tracial von Neumann algebra with Property (T) cannot embed into a tracial von Neumann algebra with Haagerup's property (\cite{ConnesJones}, Proposition 16.2.3 of \cite{anantharaman-popa}).


\begin{thm}[\cite{popa2006strong}]
\label{thm-popa-fundamental}
Let $P,Q$ be von Neumann subalgebras of a tracial von Neumann algebra $(M,\tau).$ Then the following are equivalent:

\begin{enumerate}
    \item[(a)] There is no net $(u_i)$ of unitary elements in $P$ such that for every $x,y\in M,$ $\lim_i \|E_Q(x^*u_iy)\|_2=0$.
    \item[(b)] There exists an integer $n\geq 1$, a projection $q\in M_n(Q)$, a nonzero partial isometry $v\in M_{1,n}(M)$ and a normal unital homomorphism $\theta\colon P\to qM_n(Q)q$ such that $v^*v\leq q$ and $xv = v\theta(x)$ for all $x\in P$.
\end{enumerate}
\end{thm}

When one of the above equivalent conditions is satisfied, we write $P\prec_M Q$. The following is well known and  follows immediately from the definition above and \cite{ConnesJones}. 

\begin{prop}\label{intertwine-(H)-(T)}
    If $A,B\subset (M,\tau)$ where $A$ is a von Neumann subalgebra with Haagerup's property, and $B$ is a diffuse von Neumann subalgebra with Property (T), then $B\not\prec_M A.$
\end{prop}

\begin{thm}[Theorem 1.1 of \cite{IPP}]\label{IPP}
    Let $B\subset M_1,M_2$ be a common von Neumann subalgebra of two finite von Neumann algebras $M_1$ and $M_2$. Let $M = M_1 *_B M_2$ and $Q\subset M_1$ a von Neumann subalgebra. If $Q\not\prec_M B$ then $Q'\cap M \subset M_1.$
\end{thm}

\subsection{Ultraproducts}

Let $\mathcal{U}$ be an ultrafilter on a set $I$ and $(M_i,\tau_i)_{i\in I}$ a collection of tracial von Neumann algebras. We denote by $\prod_{i\to \cU} M_i$ the \emph{tracial ultraproduct},  i.e, the quotient of $\{(x_i)_{i\in I} : \sup_{i\in I} \|x_i\| < \infty\}
$ by the closed ideal $\mathcal{J}\subset\{(x_i)_{i\in I} : \sup_{i\in I} \|x_i\| < \infty\}$  consisting of $x=(x_i)_{i\in I}$ with $\lim\limits_{i\rightarrow\mathcal U}\|x_i\|_2= 0$. If all $M_i$ are isomorphic for $i\in I,$ then $M \cong M_i$ admits a natural diagonal inclusion $\iota:M\to \prod_{i\to \cU} M_i$ given by $\iota(x) = (x)_{i\in I}$. For notational simplicity we identify $M$ with $\iota(M).$ In the case where all $M_i$ are isomorphic to $M$ we write $M^\cU := \prod_{i\to \cU} M_i.$ If $(M,\tau)$ is a \twoone factor, we say that $M$ has \emph{Property Gamma} if $M'\cap M^\cU \neq \C1$.

\begin{prop}[Proposition 5.4.1 of \cite{anantharaman-popa}]
    If $(M_n,\tau_n)$ is a sequence of tracial factors such that $\lim_n \dim M_n = \infty$ and $\cU$ is a free ultrafilter on $\N$ then $\prod_\cU M_n$ is a \twoone factor.    
\end{prop}
The following is well-known to experts. We  include a proof for reader's convenience. Recall that a projection $p\in M$ is called minimal if $pMp = \C p$ and central if $px=xp$ for all $x\in M.$
\begin{prop}\label{ultra classificaiton}
    Let $(M_i,\tau_i)$ be tracial von Neumann algebras. Then $\prod_{i\to \mathcal{U}} M_i$ is a factor if and only if there exists a sequence of either minimal  projections in the center or zero projections $p_i\in M_i$ with $\tau_i(p_i)\to_{\mathcal{U}} 1$. Furthermore, it is a \twoone factor if and only if there exists a sequence of either minimal central projections or zero projections $p_i\in M_i$ with $\tau_i(p_i)\to_{\cU} 1$ and that moreover $\dim(p_iM_ip_i) \to_{\mathcal{U}} \infty$.
\end{prop}
 \begin{proof}
   If there exists a sequence of either minimal central projections or zero projections $p_i\in M_i$ with $\tau_i(p_i)\to_{\mathcal{U}} 1$, consider $ x = (x_i) = (x_ip_i) \in \prod_{i\to \mathcal{U}} M_i$ and follow the proof of Proposition 5.4.1 of \cite{anantharaman-popa}, we have that $\prod_{i\to \mathcal{U}} M_i$ is a factor. The same argument proves the backward direction of the furthermore part. 
   
   Conversely, if $\prod_{i\to \mathcal{U}} M_i$ is a factor, then for any $p \in \cP(\cZ(\prod_{i\to \mathcal{U}} M_i))$, $p = 0$ or $1$. Then for all $p_i \in \cP(\cZ(M_i)) $, $\tau_i(p_i)(1-\tau_i(p_i))\to_{\mathcal{U}} 0$. Let $I_0 = \{i: \textrm{ there exists a minimal projection } p_i \in \cP(\cZ(M_i)) \textrm{ with } \tau_i(p_i) > 1/3\}$. Then for $i \notin I_0$, there exists $p_i \in \cP(\cZ(M_i)) $ with $1/3 \le \tau_i(p_i) < 2/3$ and $\tau_i(p_i)(1-\tau_i(p_i)) \ge 2/9$. Therefore $I_0 \in \mathcal{U}$. Take $p_i $ to be a minimal projection in $\cZ(M_i)$ with $\tau_i(p_i) > 1/3 $ for $i \in I_0$ and $p_i = 0$ otherwise. Then by $\tau_i(p_i)(1-\tau_i(p_i))\to_{\mathcal{U}} 0$, for any small enough $\varepsilon > 0$, $I_{\varepsilon}:= \{i: \tau_i(p_i)>1- \varepsilon\}\supseteq I_0 \cap \{i: \tau_i(p_i)(1-\tau_i(p_i))< \frac{\varepsilon}{3}\}$ belongs to $\mathcal{U}$. This means the sequence of either minimal central projections or zero projections $p_i\in M_i$ satisfies $\tau_i(p_i)\to_{\mathcal{U}} 1$.

   Finally, for the forward direction of the furthermore part, assume to the contrary that any sequence of either minimal central projections or zero projections $p_i\in M_i$ with $\tau(p_i) \to 1$ satisfies $\dim(p_iM_ip_i) \not\to_{\mathcal{U}} \infty$. Since $\prod_{i\to \mathcal{U}} M_i$ is a factor, by what has been proved we may fix a sequence of either minimal central projections or zero projections $p_i\in M_i$ with $\tau(p_i) \to 1$. Then $\dim(p_iM_ip_i) \to_{\mathcal{U}} N$ for some natural number $N$, and therefore $J = \{i: \dim(p_iM_ip_i) = N\} \in \mathcal{U}$. But then $p_iM_ip_i$, being a factor, is necessarily $\M_{\sqrt{N}}(\C)$. Thus, $\prod_{i\to \mathcal{U}} M_i = \prod_{i\to \mathcal{U}} p_iM_ip_i = \M_{\sqrt{N}}(\C)$, a contradiction.
 \end{proof}
 

\begin{defn}\label{cofinal}
    An ultrafilter $\mathcal U$ on a set $S$ is called \emph{countably incomplete} if there exists a sequence $(A_n)_{n\in\N}$ of sets in $\mathcal U$ such that $\cap_{n\geq 1} A_n = \emptyset.$ Otherwise, $\mathcal U$ is called \emph{countably complete}. 
\end{defn}

The following fact, for the case where the algebra involved is separable, is contained in Lemma 2.3(2) of \cite{BCI15}. However, since we will apply the ultrapower construction to algebras with density characters up to the continuum, we need a slightly stronger result.

\begin{lem}\label{countably-complete}
    If $(M,\tau)$ is a tracial von Neumann algebra with density character below the first uncountable measurable cardinal (in particular, if $M$ has density character less than or equal to the continuum), and $\cU$ is countably complete then the diagonal embedding $\Delta_M: M \to M^{\cU}$ is a trace-preserving $\ast$-isomorphism.
\end{lem}

\begin{proof}
    By Proposition 4.2.7 of \cite{CK90}, $\cU$ is $\alpha$-complete where $\alpha$ is an uncountable measurable cardinal; i.e., if $\{A_i\}_{i \in I} \subset \cU$ is a collection of sets with $|I| < \alpha$, then $\cap_i A_i \in \cU$. Since $M$ has density character below the first uncountable measurable cardinal, there is a $\|\cdot\|_2$-dense subset $\{z_j\}_{j \in J}$ with $|J| < \alpha$. The proof then proceeds exactly the same as the proof of Lemma 2.3(2) of \cite{BCI15}.

    For the in particular part, we note that by Theorem 4.2.14 of \cite{CK90}, the first uncountable measurable cardinal $\alpha$ is inaccessible, so in particular $\beta < \alpha$ implies $2^\beta < \alpha$. Since $\alpha$ is uncountable, $\aleph_0 < \alpha$, so $2^{\aleph_0} < \alpha$.
\end{proof}

\begin{thm}[Corollary 3.8 of \cite{ScottSri2019ultraproduct}]\label{scott-sri-thm}
    Let $(M_n)_n$ be a sequence of \twoone factors. Suppose that $N$ is a separable von Neumann algebra that embeds into $R^\cU$, where $\cU$ is an ultrafilter. Then $N$ is amenable if and only if for any two embeddings $\pi,\rho:N\to M = \prod_\cU M_n$ there is a unitary $u\in M$ such that $\pi(x) = u\rho(x)u^*$ for all $x\in N.$
\end{thm}

\begin{defn}
    If $\mathcal U$ is an ultrafilter on $S$ and $\mathcal V$ is an ultrafilter on $T$ then $\mathcal U \rtimes \mathcal V$ is defined by $X \in\cU\rtimes\cV $ if and only if $\{s : \{t: (s,t) \in X\} \in \cV\}\in \cU$. $\cU\rtimes \cV$ is an ultrafilter on $S\times T$.
\end{defn}

The following lemma appears as Theorem 2.1 in \cite{capraro2012product}.

\begin{lem}
    If $(M,\tau)$ is a tracial von Neumann algebra and $\cU,\cV$ are ultrafilters on sets $S,T$ respectively then $(M^\mathcal U)^ \mathcal V \cong M^{\mathcal U \rtimes \mathcal V}$.
\end{lem}

\subsection{Model Theory}

We say two tracial von Neumann algebras $(M,\tau_M)$ and $(N,\tau_N)$ are \emph{elementarily equivalent} if there exists an ultrafilter $\cU$ (possibly on an uncountable set) such that $M^\cU\cong N^\cU.$ Note that by Lemma \ref{countably-complete}, we may assume $\cU$ is countably incomplete when the algebras involved have density characters at most continuum. Let $\Delta_M:M\to M^\cU$ denote the diagonal embedding. An embedding of tracial von Neumann algebras $\iota :M \to (N,\tau)$ is said to be an \emph{elementary embedding} if there is an ultrafilter $\cU$ and an isomorphism $\Phi : M^\cU \to N^\cU$ such that $\Phi \circ \Delta_M = \Delta_N \circ \iota.$ When $M\subset N$ is an inclusion of tracial von Neumann algebras such that the inclusion map is an elementary embedding, we write $M\preceq N$ and say $M$ is an \emph{elementary substructure} of $N$.

The following can be found as Theorem 2.3 of \cite{FHS2014b}.

\begin{thm}[Downward Löwenheim-Skolem]\label{DLS}
    Let $(M,\tau)$ be a tracial von Neumann algebra. Let $X\subset M$ be a separable subset.
    Then there exists $M_0\preceq M$ such that $X\subset M_0$ and $M_0$ is a separable tracial von Neumann algebra.
\end{thm}

Assuming the Continuum Hypothesis, all results in this paper which apply to ultrapowers of \twoone factors also apply to ultraproducts of matrices.

\begin{prop}
    Assume the Continuum Hypothesis. Then every matrix ultraproduct $\cQ = \prod_\cU M_{k(n)}(\C)$ such that $\lim_\cU k(n) = \infty$ and $\cU$ is a free ultrafilter on $\N$ is isomorphic to an ultrapower of a \twoone factor.
\end{prop}

\begin{proof}
    By Proposition 4.11 of \cite{FHS2014a}, $\cQ$ is countably saturated (i.e., $\aleph_1$-saturated). Since the density character of $\cQ$ is $2^{\aleph_0} = \aleph_1,$ $\cQ$ is saturated. By Downward Löwenheim-Skolem, there is a separable tracial factor $N$ such that $N$ is elementarily equivalent to $\cQ$. Again by Proposition 4.11 of \cite{FHS2014a}, $N^\cU$ is countably saturated and therefore saturated. Since $N^\cU$ and $\cQ$ have the same density character and are both saturated, they are isomorphic by Proposition 4.13 of \cite{FHS2014a}. It is clear $N$ is a \twoone factor since $\cQ$ is not elementarily equivalent to a matrix algebra.
\end{proof}

\begin{defn}
    A tracial von Neumann algebra $M$ is \emph{existentially closed} (e.c.) in the class of tracial von Neumann algebras if for every tracial von Neumann algebra $N \supset M$, there is an ultrafilter $\cU$ such that $M\subset N\subset M^\cU$, where the inclusion of $M$ is the diagonal embedding of $M$ in $M^\cU.$
\end{defn}

We note that e.c. tracial von Neumann algebras are automatically \twoone factors \cite{goldbring2013theory}.

\begin{lem}
    A separable tracial von Neumann algebra $M$ is e.c. if and only if for every separable tracial von Neumann algebra factor $N \supset M$, there is an ultrafilter $\cU$ such that $N$ embeds in $M^\cU$ and the embedding commutes with the diagonal embedding of $M$ in $M^\cU.$
\end{lem}

\begin{proof}
    It suffices to assume that for every separable tracial von Neumann algebra $N \supset M$, there is an ultrafilter $\cU$ such that $N$ embeds in $M^\cU$ and the embedding commutes with the diagonal embedding of $M$ in $M^\cU$ and prove that $M$ is e.c., since the other direction is immediate.
    
    Let $M\subset N$ be an inclusion of tracial von Neumann algebras. Then by Downward Löwenheim-Skolem there is an elementary separable substructure $N_0$ of $N$ containing $M.$ By hypothesis, there is an ultrafilter $\cU$ and an inclusion $f:N_0 \hookrightarrow M^\cU$ such that  $f$ restricts to the diagonal embedding $\Delta_{M,\cU}$ on $M$. Since $N_0 \preceq N,$ there is an ultrafilter $\cV$ and an inclusion $g:N\hookrightarrow N_0^\cV$ such that $g$ restricts to the diagonal embedding $\Delta_{N_0,\cV}$ on $N_0$. Lastly, since there is an inclusion $N_0\subset M^\cU$ there is also an induced inclusion $f^\cV:N_0^\cV \hookrightarrow (M^\cU)^\cV = M^{\cU\rtimes\cV}$. The composition of inclusions $M\subset N\subset N_0^\cV \subset M^{\cU\rtimes\cV}$ is then equal to the diagonal embedding $\Delta_{M,\cU\rtimes\cV}$. In other words, the diagram below commutes.

    \[
    \begin{tikzcd}
    M \ar[r, hook] \ar[dr, hook] & N_0 \ar[r, hook] \ar[d,hook,"f"] \ar[dr,hook] & N \ar[d, hook, "g"]\\
    &M^\cU \ar[dr,hook]& N_0^\cV \ar[d,hook,"f^\cV"]  \\
    & &M^{\cU\rtimes\cV} 
    \end{tikzcd}
    \]
\end{proof}

\begin{lem}\label{e.c. lifting}
    If $M$ is a tracial von Neumann algebra such that all separable elementary substructures of $M$ are e.c., then $M$ is e.c.
\end{lem}

\begin{proof}
    This follows from Downward Löwenheim-Skolem and the fact that every inductive limit of a chain of e.c. tracial von Neumann algebras is e.c. (Fact 2.3.2(4) of \cite{AGKE}).
\end{proof}

\section{Proof of main results}\label{construction section}

\subsection{Conjugating subalgebras}\label{hnn-ext}

We follow the notation of Section 3 of \cite{fimavaes} (see also \cite{uedahnn}). Let $(M,\tau)$ be a tracial von Neumann algebra, $A\subset M$ a von Neumann subalgebra, and $\theta:A\to M$ a trace-preserving *-homomorphism. Then there is a tracial von Neumann algebra $P$ which contains $M$ and is generated by $M$ and a single additional Haar unitary $u$ such that $uxu^* = \theta(x)$ for all $x\in A$. The algebra $P$ is called the HNN extension of $M$ with respect to $A$ and $\theta,$ and denoted $\mathrm{HNN}(M,A,\theta).$ Write $A_1 = A$ and $A_{-1} = \theta(A).$ An element $x = x_0u^{\varepsilon_1}\cdots u^{\varepsilon_n}x_n$ with $x_i\in M$ and $\varepsilon_i\in\{-1,1\}$ is said to be reduced if $x_i\in M\ominus A_{\varepsilon_i}$ whenever $\varepsilon_i\neq \varepsilon_{i+1}$. By convention elements of $M\ominus \C1$ are also considered reduced in $P.$ The reduced words SOT-densely span $P.$ Then $P$ satisfies the following universal property:

\begin{prop}[Proposition 3.2 of \cite{fimavaes}]\label{hnn-univ-prop}
    Let $P = \mathrm{HNN}(M,A,\theta)$ be an HNN extension. Assume that $(Q,\tau_Q)$ is any tracial von Neumann algebra, that $\pi:M\to Q$ is a trace-preserving embedding and that $w\in Q$ is a unitary satisfying
    \begin{itemize}
        \item $\pi(\theta(x)) = w\pi(x)w^*$ for all $x\in A,$
        \item for all reduced $x = x_0u^{\varepsilon_1}\cdots u^{\varepsilon_n}x_n \in P,$ we have $\tau_Q(\pi(x_0)w^{\varepsilon_1}\cdots w^{\varepsilon_n}\pi(x_n)) = 0.$
    \end{itemize}
    Then there exists a unique trace-preserving *-homomorphism $\Tilde{\pi}:P\to Q$ extending $\pi$ and satisfying $\Tilde{\pi}(u) = w.$
\end{prop}

It is known that the HNN extension described above is isomorphic to a corner of an amalgamated free product (see Proposition 3.1 of \cite{ueda2008remarks}). Our contribution is to show that the HNN extension is a natural unital subalgebra of an amalgamated free product, allowing us to easily apply intertwining results such as found in \cite{IPP}.

\begin{defn}
    
We proceed with our construction by letting $(M, \tau_M)$ be a tracial von Neumann algebra, $A \subseteq M$ be a subalgebra, and $\theta: A \to M$ be a trace-preserving $\ast$-homomorphism. 

Let $A \wr^\ast \Z$ be the free wreath product of $A$ with $\Z$, i.e., $A \wr^\ast \Z = (\ast^{n \in \Z} A_n) \rtimes \Z$ where $A_n$ are copies of $A$ and $\Z$ acts on $\ast^{n \in \Z} A_n$ by permuting $A_n$, i.e., the conjugation by the generator of $\Z$ sends $A_n$ to $A_{n + 1}$. For notational purposes, let $B_1$, $B_2$, and $B_3$ be three copies of $A \wr^\ast \Z$, in which the copies of $A$ are labeled $A^1_n$, $A^2_n$, and $A^3_n$, respectively, and in which the generators of $\Z$ are labeled $t$, $s$, and $r$, respectively.
    
We now let $M_1 = M \ast_A B_1$, where the inclusion of $A$ into $B_1$ sends $A$ to $A^1_0$. Then, let $M_2 = M_1 \ast_{\theta(A)} B_2$, where the inclusion of $\theta(A)$ into $B_2$ first sends $\theta(A)$ to $A$ via $\theta^{-1}$, then sends $A$ to $A^2_0$.

Now, we observe that in $M_1$, $A^1_1$ is orthogonal to $M$, so in particular it is orthogonal to $\theta(A)$. We also have $A^2_1$ is orthogonal to $A^2_0$. Hence, $A^1_1$ and $A^2_1$ are freely independent in $M_2$, so $M_2 \supseteq A^1_1 \vee A^2_1 \cong A \ast A$. Thus, we may let $M_3 = M_2 \ast_{A^1_1 \vee A^2_1} B_3$, where $A^1_1$ is sent to $A^3_0$ and $A^2_1$ is sent to $A^3_1$. Let $w = s^{-1}rt$. We easily see that $wxw^\ast = \theta(x)$ for all $x \in A$.

We define $\Phi(M,A,\theta) = \langle M, w\rangle''\subset M_3$. To ease notation, we set $N = \Phi(M,A,\theta).$ For an element $x\in M_3,$ we define $ \mathring{x} = x-\tau(x).$ 

\end{defn}

\begin{thm}\label{our hnn}
     Let $(M,\tau)$ be a tracial von Neumann algebra, $A\subset M$ a von Neumann subalgebra and $\theta:A\to M$ an embedding. Then $\mathrm{HNN}(M,A,\theta) \cong \Phi(M,A,\theta)$.
\end{thm}

\begin{proof}
    The aim of the proof is to apply Proposition \ref{hnn-univ-prop}. There is clearly an embedding of $M$ as a subalgebra into $\mathrm{HNN}(M,A,\theta)$, and it is also immediately clear from the construction that $\theta(x) = wxw^*$ for all $x\in A.$ Since $\mathrm{HNN}(M,A,\theta)$ is generated by $M$ and $w,$ if we can show that all reduced words of the form $x_0w^{\varepsilon_1}\cdots w^{\varepsilon_n}x_n$ have trace 0, then we will have proved the theorem.

    We now make a series of four technical claims which will allow us to prove that the trace of all reduced words is equal to 0.

    We note before proving the claims that $t\perp M.$ Indeed, since $t$ is a non-trivial group unitary in the crossed product $B_1 = *_n A_n^1 \rtimes \Z$, we have $\tau(at) = 0$ for all $a\in A_1^1.$ That is, $t \perp A_1^1 = A.$ But then $t$ is a reduced word in $B_1 *_a M,$ and therefore is orthogonal to $M.$ 

    We note too that $s^{-1}b \perp A_0^2$ for any $b\in A_1^2$ for a similar reason. The unitary $s^{-1}$ is a non-trivial group unitary in the crossed product $B_2 = *_n A_n^2 \rtimes \Z$, and so we have $\tau(s^{-1}c) = 0$ for all $c\in A_0^2 \vee A_1^2.$ In particular, $\tau(s^{-1}bd) = 0$ for all $d\in A_0^2.$
    
    We also remark that for any $x\in M$, $tx \perp \theta(A)$. Indeed, as $xa \in M$ for any $a \in \theta(A)$ and $t \perp M$, we have $\tau(txa) = \tau(t)\tau(xa) = 0$.

\begin{claim}\label{claim I}
    $txs^{-1} \perp A^1_1 \vee A^2_1$ for any $x \in M$.
\end{claim}
    
\begin{proof}[Proof]\renewcommand{\qedsymbol}{$\blacksquare$}
    We first note that, by construction, $A^1_1$ and $A^2_1$ are freely independent. Therefore they are 2-norm densely spanned by elements of the form  $a_1b_1 \cdots b_{m - 1}a_m$ where $a_1, a_m \in A^1_1$, $a_2, \cdots a_{m - 1} \in A^1_1 \ominus \C$, and $b_1, \cdots, b_{m - 1} \in A^2_1 \ominus \C$.

    To show that $txs^{-1} \perp A^1_1 \vee A^2_1$, it therefore suffices to show $\tau(txs^{-1}a_1b_1 \cdots b_{m - 1}a_m) = 0$ for all $a_1, a_m \in A^1_1$, $a_2, \cdots a_{m - 1} \in A^1_1 \ominus \C$, and $b_1, \cdots, b_{m - 1} \in A^2_1 \ominus \C$.
    
     If $\tau(a_1) = 0$, then,
    \begin{equation*}
        txs^{-1}a_1b_1 \cdots b_{m - 1}a_m = (tx)s^{-1}a_1b_1 \cdots b_{m - 1}\mathring{a_m} + \tau(a_m)(tx)s^{-1}a_1b_1 \cdots b_{m - 1}
    \end{equation*}
    is a reduced word in $M_2 = M_1 \ast_{\theta(A)} B_2$ as $tx \perp \theta(A)$, and therefore has trace 0. In the case $\tau(a_1) \neq 0$, we similarly have,
    \begin{equation*}
    \begin{split}
        txs^{-1}a_1b_1 \cdots b_{m - 1}a_m =& (tx)s^{-1}\mathring{a_1}b_1 \cdots b_{m - 1}a_m + \tau(a_1)(tx)(s^{-1}b_1) \cdots b_{m - 1}a_m\\
        =& (tx)s^{-1}\mathring{a_1}b_1 \cdots b_{m - 1}a_m + \tau(a_1)(tx)(s^{-1}b_1)a_2 \cdots b_{m - 1}\mathring{a_m}\\
        &\hspace{3cm}+ \tau(a_1)\tau(a_m)(tx)(s^{-1}b_1)a_2 \cdots b_{m - 1}
    \end{split}
    \end{equation*}
    
    The first term has trace zero, as we have already seen. The second and third terms are reduced words, as $s^{-1}b_1 \perp A^2_0$.
\end{proof}

\begin{claim}\label{claim II}
    $sxt^{-1} \perp A^1_1 \vee A^2_1$ for any $x \in M$.
\end{claim}

\begin{proof}\renewcommand{\qedsymbol}{$\blacksquare$}
    We first note that for $x\in M$ and $a\in A_1^1,$ $xt^{-1}a \perp M.$ Indeed, we note that $t^{-1}\perp M$ and $a = tbt^{-1}$ for some $b\in A_0^1 \subset M$ and so for any $y\in M,$ $\tau(xt^{-1}ay) = \tau(xbt^{-1}y) = \tau(t^{-1})\tau(yxb) = 0.$ In particular, we have that $xt^{-1}a \perp \theta(A).$

    Now, for $a_1, a_m \in A^1_1$, $a_2, \cdots a_{m - 1} \in A^1_1 \ominus \C$, and $b_1, \cdots, b_{m - 1} \in A^2_1 \ominus \C$, we have that
    $$ sxt^{-1}a_1b_1 \cdots b_{m - 1}a_m = \tau(a_m)s(xt^{-1}a_1)b_1\cdots b_{m-1} + s(xt^{-1}a_1)b_1\cdots b_{m-1}\mathring{a_m}, $$
    which is a sum of two terms which are reduced words in $M_2 = M_1 *_{\theta(A)} B_2,$ and therefore has trace 0. 
\end{proof}

\begin{claim}\label{claim III}
    $txt^{-1} \perp A^1_1 \vee A^2_1$ for any $x \in M\ominus A$.
\end{claim}

\begin{proof}\renewcommand{\qedsymbol}{$\blacksquare$}
    We first note that for $x\in M\ominus A,$ $txt^{-1} \perp M.$ Indeed, if $y\in M$ then $txt^{-1}y = txt^{-1}E_A(y) + txt^{-1}(y-E_A(y))$, a sum of two terms which are reduced words in $M_1 = M *_A B_1$. Therefore $\tau(txt^{-1}y) = 0.$

    We next observe that for $x\in M\ominus A$ and $a\in A_1^1,$ $txt^{-1}a \perp \theta(A).$ We note that $a = tbt^{-1}$ for some $b\in A_0^1 = A,$ and $bx\in M\ominus A,$ so that $txt^{-1}a = txbt^{-1} \perp M \supset \theta(A)$.

    Therefore, for $x\in M\ominus A$, $a_1, a_m \in A^1_1$, $a_2, \cdots a_{m - 1} \in A^1_1 \ominus \C$, and $b_1, \cdots, b_{m - 1} \in A^2_1 \ominus \C$, we have that
    $$ txt^{-1}a_1b_1 \cdots b_{m - 1}a_m = \tau(a_m)(txt^{-1}a_1)b_1\cdots b_{m-1} + (txt^{-1}a_1)b_1\cdots b_{m-1}\mathring{a_m}, $$
    which is a sum of two terms which are reduced words in $M_2 = M_1 *_{\theta(A)} B_2,$ and therefore has trace 0. 
\end{proof}

\begin{claim}\label{claim IV}
    $sxs^{-1} \perp A^1_1 \vee A^2_1$ for any $x \in M\ominus \theta(A)$.
\end{claim}

\begin{proof}\renewcommand{\qedsymbol}{$\blacksquare$}
    We first note that for $x\in M\ominus\theta(A)$, $sxs^{-1}\perp M.$ Indeed, if $y\in M$ then $sxs^{-1}y = sxs^{-1}E_{\theta(A)}(y) + sxs^{-1}(y-E_{\theta(A)}(y))$, a sum of two terms which are reduced words in $M_2 = M_1 *_{\theta(A)} B_2$. Therefore $\tau(sxs^{-1}y) = 0.$

    We next observe that for $x\in M\ominus \theta(A)$ and $a\in A_1^2,$ $sxs^{-1}a \perp \theta(A).$ We note that $a = sbs^{-1}$ for some $b\in A_0^2 = \theta(A),$ and $bx\in M\ominus \theta(A),$ so that $sxs^{-1}a = sxbs^{-1} \perp M \supset \theta(A)$.

    Therefore, for $x\in M\ominus \theta(A)$, $a_1, a_m \in A^1_1$, $a_2, \cdots a_{m - 1} \in A^1_1 \ominus \C$, and $b_1, \cdots, b_{m - 1} \in A^2_1 \ominus \C$, we have that
    \begin{equation*}
    \begin{split}
        sxs^{-1}a_1b_1 \cdots b_{m - 1}a_m =& (sxs^{-1})\mathring{a_1}b_1 \cdots b_{m - 1}\mathring{a_m} + \tau(a_1)(sxs^{-1}b_1) \cdots b_{m - 1}\mathring{a_m}\\
        +& \tau(a_m)(sxs^{-1})\mathring{a_1}b_1 \cdots b_{m - 1} + \tau(a_1)\tau(a_m)(sxs^{-1}b_1)a_2 \cdots b_{m - 1},
    \end{split}
    \end{equation*}
    which is a sum of four terms which are reduced words in $M_2 = M_1 *_{\theta(A)} B_2,$ and therefore has trace 0. 
\end{proof}

Now, let $x_0u^{\varepsilon_1}\cdots u^{\varepsilon_n}x_n$ with $x_i\in M$, $\varepsilon_i\in\{-1,1\}$, and  $x_i\in M\ominus A_{\varepsilon_i}$ whenever $\varepsilon_i\neq \varepsilon_{i+1}$ be a reduced word in $P = \mathrm{HNN}(M,A,\theta).$ We wish to show that $\tau_N(x_0w^{\varepsilon_1}\cdots w^{\varepsilon_n}x_n) = 0$ where $N = \Phi(M,A,\theta).$ Write $w = s^{-1}rt$ and $w^{-1} = t^{-1}r^{-1}s$ as in the above construction. Then $x_0w^{\varepsilon_1}\cdots w^{\varepsilon_n}x_n$ is a product of terms of the form $r,$ $r^{-1}$, $s^{-1}xt$, $t^{-1}xs$, $tyt^{-1}$, and $szs^{-1}$ where $x\in M,$ $y\in M\ominus A,$ and $z\in M\ominus\theta(A)$. Moreover the product alternates between terms of the form $r^{\pm1}$ and terms with an $s$ or a $t.$ But $B_3 \ni r,r^{-1} \perp A_1^1 \vee A_1^2$ while the other terms are in $M_2$ and orthogonal to $A_1^1\vee A_1^2.$ Therefore the expression $x_0w^{\varepsilon_1}\cdots w^{\varepsilon_n}x_n$ (possibly needing to conjugate by either $s$ or $t$) is a reduced word in $M_3 = M_2 *_{A_1^1\vee A_1^2} B_3$, and therefore has trace 0 in $N.$
\end{proof}

We record some more facts about $\mathrm{HNN}(M,A,\theta)$ which follow from the above theorem and its proof.

\begin{cor}
    In $\mathrm{HNN}(M,A,\theta),$ $E_M(w^n) = 0$ for all $n\neq 0,$ where $w$ is as in the construction given in this section.
\end{cor}

\begin{proof}
    We show that $\tau(w^n x) = 0$ for any $n > 0$ and $x \in M$. Indeed, we have,
    \begin{equation*}
        \tau(w^n x) = \tau([s^{-1}rt]^n x) = \tau(r[(ts^{-1})r]^{n-1}txs^{-1})
    \end{equation*}
    
    We observe that $r[(ts^{-1})r]^{n-1}txs^{-1}$ is a reduced word in $M_3 = M_2 \ast_{A^1_1 \vee A^2_1} B_3$, as $r \perp A^3_0 \ast A^3_1 = A^1_1 \vee A^2_1$, $ts^{-1} \perp A^1_1 \vee A^2_1$ per Claim \ref{claim I}, and $txs^{-1} \perp A^1_1 \vee A^2_1$ per Claim \ref{claim I}.
\end{proof}

\begin{cor}\label{hnn-with-(T)}
    Let $N = \mathrm{HNN}(M,A,\theta).$ Let $\tilde{A}$ be an isomorphic copy of $A.$ Assume $Q\subset M$ is such that $Q\not\prec_{\Tilde{M_1}} \tilde{A}$ for any tracial von Neumann algebra $\Tilde{M_1}$ containing $M$ and $\tilde{A}$ and $Q\not\prec_{\Tilde{M_2}} \tilde{A}*\tilde{A}$  for any tracial von Neumann algebra $\Tilde{M_2}$ containing $M$ and $\tilde{A}*\tilde{A}$. Then $Q'\cap N \subset M$.
    
    Moreover, if $M$ is a \twoone factor, $Q \subseteq M$ is an irreducible subfactor with property (T), and $A$ has Haagerup's property, then $Q' \cap N = \C$.
\end{cor}

\begin{proof}
    The result follows from Theorem \ref{IPP} since $Q'\cap N \subset Q'\cap M_3$ and
    \begin{equation*}
        Q' \cap M_3 \subseteq Q' \cap M_2 \subseteq Q' \cap M_1 \subseteq Q' \cap M.
    \end{equation*}

    For the moreover, since both $A$ and $A_1^1 \vee A^2_1 \cong A\ast A$ both have Haagerup's property, Proposition \ref{intertwine-(H)-(T)} says that every diffuse subalgebra $Q$ with Property (T) satisfies $Q\not\prec_{\Tilde{M_1}} A$ and $Q\not\prec_{\Tilde{M_2}} A*A.$ Therefore $Q'\cap N \subset Q'\cap M = \C1$. 
\end{proof}

\subsection{Inductive constructions}\label{inductive-section}
\begin{thm}\label{inductive-argument}
    For any diffuse tracial von Neumann algebra $(M, \tau_M)$ of density character at most $2^{\aleph_0}$, there exists a tracial von Neumann algebra $(N, \tau_N)$ containing $M$, of density character at most $2^{\aleph_0}$, such that whenever $A$ is a separable tracial von Neumann algebra and $\pi_1, \pi_2: A \to N$ are two embeddings, then there exists a unitary $u \in U(N)$ such that $u\pi_1(a)u^\ast = \pi_2(a)$ for all $a \in A$.
\end{thm}

\begin{proof}
    Fix a bijection $\sigma = (\sigma_1, \sigma_2): 2^{\aleph_0} \to 2^{\aleph_0} \times 2^{\aleph_0}$ such that $\sigma_1(\alpha) \leq \alpha$ for every $\alpha < 2^{\aleph_0}$. We construct an increasing sequence of algebras $(M_\alpha)_{\alpha < 2^{\aleph_0}}$, all of density character at most $2^{\aleph_0}$, by induction. Let $M_0 = M$. Assume all $M_\lambda$ with $\lambda < \alpha$ have been constructed. Let $\{(A^{\lambda, \kappa}, \pi_1^{\lambda, \kappa}, \pi_2^{\lambda, \kappa})\}_{\kappa < 2^{\aleph_0}}$ be an enumeration of all separable tracial von Neumann algebras $A^{\lambda, \kappa}$ and pairs of embeddings of $A^{\lambda, \kappa}$ into $M_\lambda$. If $\alpha = \lambda + 1$, we let $M_\alpha = \mathrm{HNN}(M_\lambda, \pi_1^{\sigma(\lambda)}(A^{\sigma(\lambda)}), \pi_2^{\sigma(\lambda)} \circ (\pi_1^{\sigma(\lambda)})^{-1})$. If $\alpha$ is a limit ordinal, let $M_\alpha = \overline{\cup_{\lambda < \alpha} M_\lambda}$ where the closure is under the SOT topology on the GNS Hilbert space associated with the trace on $\cup_{\lambda < \alpha} M_\lambda$. Finally, let $N = \cup_{\alpha < 2^{\aleph_0}} M_\alpha$. Since in a tracial von Neumann algebra, the SOT topology on the operator norm unit ball is induced by the 2-norm metric, and as the cofinality of $2^{\aleph_0}$ is larger than $\aleph_0$, we see that the operator norm unit ball of $N$ is already SOT-closed, whence it is a tracial von Neumann algebra. For any separable $A$ and embeddings $\pi_1, \pi_2: A \to N$, again because the cofinality of $2^{\aleph_0}$ is larger than $\aleph_0$, the ranges of both embeddings must already be contained in $M_\alpha$ for some $\alpha < 2^{\aleph_0}$. Thus, $(A, \pi_1, \pi_2) = (A^{\alpha, \kappa}, \pi_1^{\alpha, \kappa}, \pi_2^{\alpha, \kappa})$ for some $\kappa < 2^{\aleph_0}$, so by construction, they are already unitarily conjugate in $M_{\sigma^{-1}(\alpha, \kappa) + 1}$.
\end{proof}

\begin{cor}\label{univ-alg}
    There is a tracial von Neumann algebra $(N,\tau)$ of density character $2^{\aleph_0}$ that contains a unique copy of each separable tracial von Neumann algebra up to unitary conjugacy.
\end{cor}

\begin{proof}
    There are at most continuum many separable tracial von Neumann algebras since each is countably generated and is determined by the values of traces of all *-polynomials in the generators. Hence, there is a tracial von Neumann algebra $M$ of density character $2^{\aleph_0}$ containing all separable tracial von Neumann algebras, e.g., $M$ being the tensor product of all separable tracial von Neumann algebras. Applying Theorem \ref{inductive-argument} inductively over all separable tracial von Neumann subalgebras proves the result.
\end{proof}

\begin{cor}\label{EC cor}
    Let $N$ be as in Corollary \ref{univ-alg}. Then $N$ is e.c.
\end{cor}

\begin{proof}
    Let $N_0\preceq N$ be a separable elementary substructure. We note that by Theorem \ref{DLS}, for some ultrafilter $\cU,$ $N_0 \to N \to N_0^\cU$ commutes with the diagonal embedding. Then if $M\supset N_0$ is a separable tracial von Neumann algebra, $M$ embeds into $N$ since $N$ contains every separable tracial von Neumann algebra. 

    There are now two embeddings of $N_0$ into $N$: one that passes through $M$ and the one from the beginning. By Corollary \ref{univ-alg}, these embeddings are unitarily conjugate. So we can conjugate $M$ by a unitary $u$ so that $N_0 \subset uMu^* \subset N \subset N_0^\cU$ and all inclusions commute with the diagonal embedding. Therefore $N_0$ is e.c.

    $N$ being e.c. now follows from Lemma \ref{e.c. lifting}.
\end{proof}

\begin{lem}\label{non-ultra}
    Any tracial $(N,\tau)$ with the property that it contains at least one non-amenable Connes-embeddable separable \twoone factor and that any two embeddings of this factor are unitarily conjugate is not isomorphic to an ultraproduct of \twoone factors.
\end{lem}

\begin{proof}
    By Theorem \ref{scott-sri-thm}, in an ultraproduct of \twoone factors, any non-amenable Connes-embeddable separable \twoone factor has at least 2 non-unitarily conjugate embeddings.
\end{proof}

\begin{defn}[Definition 3.1 of \cite{patchellelayavalli2023sequential}]
    Let $(M,\tau)$ be a diffuse tracial von Neumann algebra. Fix a countably incomplete ultrafilter $\cU$. Define $\sim_{M}$ to be the equivalence relation defined on $\cH(M)$ (Haar unitaries in $M$) in the following way: we say $u\sim_M v$ if there are $w_1,\ldots,w_n \in \cH(M^\cU)$ such that $[u,w_1] =[w_k,w_{k+1}] = [w_n,v] = 0$ for all $1\leq k < n$.
\end{defn}

\begin{thm}[Theorem 4.2 of \cite{CIKE23}]\label{cike-technical}
    Let $(M,\tau)$ be a diffuse tracial von Neumann algebra. Let $u_1,u_2\in\cU(M)$ be unitaries such that $\{u_1\}''\perp\{u_2\}''.$ Then there exists a \twoone factor $P = \Psi(M,u_1,u_2)$ containing $M$ such that there exist Haar unitaries $v_1,v_2\in\cH(P)$ with the property that $[u_1,v_1]=[v_1,v_2]=[v_2,u_2]=0.$ Furthermore, if $Q\subset M$ is a von Neumann subalgebra such that $Q\not\prec_M \{u_i\}''$ for $i=1$ and $i=2$ then $Q'\cap P \subset M.$ In particular, if $Q\subset M$ is a \twoone subfactor then $Q'\cap P\subset M$.
\end{thm}

We are now ready to prove Theorem \ref{intro main theorem A}, which follows immediately from the following theorem.

\begin{thm}\label{totally main theorem}
    There exists a \twoone factor $M$ that is not isomorphic to an ultraproduct of \twoone factors, is non-Gamma, and every separable subalgebra of $M$ with Haagerup's property embeds in $M$ uniquely up to unitary conjugacy. In particular, $M$ is not e.c. The factor $M$ can also be chosen to have non-positive 1-bounded entropy; i.e., $h(M) \leq 0.$
\end{thm}

\begin{proof}
    Let $M_0 = L(SL_3(\Z))$. Similar to the proof of Theorem \ref{inductive-argument}, we recursively define an increasing sequence of algebras $(M_\alpha)_{\alpha < 2^{\aleph_0}}$. Again, we fix a bijection $\sigma = (\sigma_1, \sigma_2): 2^{\aleph_0} \to 2^{\aleph_0} \times 2^{\aleph_0}$ such that $\sigma_1(\alpha) \leq \alpha$ for every $\alpha < 2^{\aleph_0}$. Assume all $M_\lambda$ with $\lambda < \alpha$ have been constructed. Let $\{Z^{\lambda,\kappa}\}_{\kappa< 2^{\aleph_0}}$ be an enumeration of all possible tuples of the form $Z^{\lambda,\kappa} = (A^{\lambda,\kappa},\pi_1^{\lambda, \kappa}, \pi_2^{\lambda, \kappa})\}_{\kappa < 2^{\aleph_0}}$ where $A^{\lambda,\kappa}$ is a separable tracial von Neumann algebra with Haagerup's property and $(\pi_1^{\lambda, \kappa}, \pi_2^{\lambda, \kappa})$ is a pair of embeddings of $A^{\lambda, \kappa}$ into $M_\lambda$, or of the form $Z^{\lambda,\kappa} = (u_1^{\lambda,\kappa},u_2^{\lambda_\kappa})$ where $u_1,u_2 \in \cU(M_\lambda)$ and $\{u_1\}''\perp \{u_2\}''.$
    
    If $\alpha = \lambda + 1$, we let 
    $$M_\alpha = \begin{cases}
        \mathrm{HNN}(M_\lambda, \pi_1^{\sigma(\lambda)}(A^{\sigma(\lambda)}), \pi_2^{\sigma(\lambda)} \circ (\pi_1^{\sigma(\lambda)})^{-1}) & \text{if } Z^{\sigma(\lambda)} = (A^{\sigma(\lambda)}, \pi_1^{\sigma(\lambda)}, \pi_2^{\sigma(\lambda)})\\
        \Psi(M_\lambda,u_1^{\sigma(\lambda)},u_2^{\sigma(\lambda)}) & \text{if } Z^{\sigma(\kappa)} = (u_1^{\sigma(\lambda)},u_2^{\sigma(\lambda)})
    \end{cases}.$$ If $\alpha$ is a limit ordinal, let $M_\alpha = \overline{\cup_{\lambda < \alpha} M_\lambda}$. Finally, let $M = \cup_{\alpha < 2^{\aleph_0}} M_\alpha$. Again, it is a tracial von Neumann algebra.

    We need to check that $M$ is a non-Gamma \twoone factor. It is clear that any two embeddings of a separable subalgebra with Haagerup's property in $M$ are unitarily conjugate. It is at this stage that Haagerup's property is crucial. By the moreover statement of Corollary \ref{hnn-with-(T)}, Theorem \ref{cike-technical}, and as $M_0$ has property (T), we have, using transfinite induction, that $M_0' \cap M = M_0' \cap M_0 = \C$. Since $M_0$ has property (T), we also have that $M_0' \cap M^\cU = (M_0' \cap M)^\cU = \C$. Therefore $M' \cap M^\cU = \C$ for any ultrafilter $\cU$ so $M$ is a non-Gamma \twoone factor. Applying Lemma \ref{non-ultra} to $L(\F_2)$, a non-amenable Connes-embeddable separable \twoone factor with Haagerup's property, shows that $M$ is not isomorphic to an ultraproduct of \twoone factors.

    We also need to check that $h(M) \leq 0.$ By Lemma 5.1 of \cite{patchellelayavalli2023sequential} it suffices to check that $\sim_M$ has a unique orbit. Indeed, let $u_1,u_2 \in \cH(M)$ be Haar unitaries. Let $\cU$ be a free ultrafilter on $\N.$ By \cite{Popaindep} there is a diffuse separable abelian von Neumann algebra generated by a Haar unitary $v\in M^\cU$ such that $\{u_1,u_2\}''$ and $\{v\}''$ are freely independent (see also \cite{popa1995free}). By Theorem 5.1 of \cite{houdayer2023asymptotic} we can lift $v$ to a sequence of unitaries $(v_n)_n$ in $M$ such that $\{u_1\}''\perp \{v_n\}''$ for each $n.$ By construction, $\Psi(M_\lambda,u_1,v_n) \subset M$ for some ordinal $\lambda < 2^{\aleph_0}$ and therefore $M$ contains Haar unitaries $w_{1,n},w_{2,n}$ such that $[u_1,w_{1,n}] = [w_{1,n},w_{2,n}] = [w_{2,n},v_n] = 0$ for all $n,$ as in Theorem \ref{cike-technical}. Therefore $u_1 \sim_{M^\cU} v$. Similarly, $u_2 \sim_{M^\cU} v,$ and so $u_1\sim_{M^\cU} u_2.$ Therefore $\sim_{M^\cU}$ has one orbit, implying by Proposition 6.2 of \cite{patchellelayavalli2023sequential} $\sim_M$ also only has one orbit.
\end{proof}
\subsection{Indecomposability results}
Item (1) and (3) of the proposition below are results of \cite{popaortho, kadisonsingerpopa} and item (2) is quite an elementary observation, not appearing before in the literature. We include the proofs for the convenience of the readers. We also point out to the reader the note \cite{hiatt2024singular}. 
\begin{prop}\label{ucc main prop}
    Let $M$ be a \twoone factor such that all pairs of Haar unitaries  are conjugate. 
    \begin{enumerate}
        \item \cite{popaortho, sorincourse} $M$ does not contain a diffuse regular von Neumann subalgebra $B$ and a Haar unitary $u$ such that $\{u\}''\perp B$. In particular, $M$ is neither the tensor product of two \twoone factors nor a crossed product of a diffuse tracial von Neumann algebra by an infinite group; $M$ does not have a Cartan subalgebra.
        \item $M$ is not an amalgamated free product $M_1 *_B M_2$ with $B\subset M_i$, $M_1$ diffuse, and where there exist unitaries $u_i \in M_i\ominus B$, $i=1,2$, moreover, $M$ is not an HNN extension of a diffuse von Neumann algebra;
        \item \cite{kadisonsingerpopa} $M$ does not have a separable MASA.
    \end{enumerate}
\end{prop}

\begin{proof} 

    (1) Suppose $M$ has a diffuse tracial subalgebra $B$ and a Haar unitary $u\in \cH(M)$ such that $\{u\}''\perp B.$ Take $v\in \cH(B).$ By the hypothesis, there is $w\in \cU(M)$ such that $wvw^* = u.$ Then $B_0 = \{v\}''$ is a diffuse subalgebra of $B$ such that $wB_0w^* \perp B.$ By Lemma \ref{popa-ortho-cor2.6}, we have that $w \perp \cN(B)''$. In particular, $\cN(B)'' \neq M,$ so $B$ cannot be regular. In a tensor product $M_1\otimes M_2$ of two \twoone factors, $M_1$ is regular and diffuse and any Haar unitary $u\in \cH(M_2)$ satisfies $\{u\}''\perp M_1$. In a crossed product of a diffuse algebra by an infinite group $B\rtimes \Gamma,$ $B$ is diffuse and regular and any Haar unitary $u\in L(\Gamma)$ again satisfies $\{u\}''\perp B.$ ($L(\Gamma)$ contains Haar unitaries since it is diffuse.)

    We also include more elementary proofs for the cases of tensor products and crossed products. First suppose that $M$ is a tensor product $M_1\otimes M_2$ such that $M_i$ are \twoone factors. Let $u$ be a Haar unitary in $M_1.$ For pure tensors $x_1\otimes x_2,y_1\otimes y_2 \in M,$ we see that $\lim_{n\to\infty} E_{1\otimes M_2}((x_1\otimes x_2)(u^n\otimes 1)(y_1\otimes y_2)) = \tau(x_1u^ny_1)x_2y_2 \to 0.$ Linearity shows the same holds for linear combinations of pure tensors. For general $x,y\in M,$ and for $\varepsilon > 0,$ take $x',y'\in M_1\odot M_2$ such that $\|x-x'\|_2 < \varepsilon$ and $\|y-y'\|<\varepsilon$. By the Kaplansky density theorem we can also choose $\|y'\|_\infty\leq \|y\|_\infty.$ Then 
    \begin{align*}
        \|E_{1\otimes M_2}(xu^ny)\|_2 &\leq \|E_{1\otimes M_2}(xu^n(y-y'))\|_2 + \|E_{1\otimes M_2}((x-x')u^ny')\|_2 + \|E_{1\otimes M_2}(x'u^ny')\|_2\\
        &\leq \|xu^n(y-y')\|_2 + \|(x-x')u^ny'\|_2 + \|E_{1\otimes M_2}(x'u^ny')\|_2\\
        &\leq \|y-y'\|_2\|xu^n\|_\infty + \|x-x'\|_2\|u^ny'\|_\infty + \|E_{1\otimes M_2}(x'u^ny')\|_2.
    \end{align*}
    This clearly goes to 0 as we choose better approximations $x',y'$ and as $n\to\infty.$ Applying the easy direction of Theorem \ref{thm-popa-fundamental} gives two non-conjugate Haar unitaries.

    The proof for crossed products is very similar: if $M = B \rtimes_\sigma \Gamma$ where $B$ is diffuse and $\Gamma$ is infinite, then $L\Gamma$ is diffuse and there are Haar unitaries in both $B$ and $L\Gamma.$ Let $u$ be a Haar unitary in $B$. Let $bu_g,cu_h$ be elements in $B\cdot L\Gamma$ (such elements densely span $B\rtimes \Gamma$). Then $E_{L\Gamma}(bu_g u^n cu_h) = E_{L\Gamma}(b\sigma_g(u^nc)u_{gh}) = \tau(b\sigma_g(u^nc)) \delta_{gh,e} = \tau(\sigma_{g^{-1}}(b)u^n c) \delta_{gh,e} \to 0$ as $n\to\infty.$
    
    (2) In an amalgamated free product $M = M_1*_B M_2$, take unitaries $u_i\in M_i\ominus B$. By the Kaplansky density theorem, it suffices to check that $\|E_{M_1}(x(u_1u_2)^ny)\|_2 \to 0$ as $n\to\infty$ for any reduced words $x,y\in M$ since reduced words SOT-densely span $M$. Denote by $X_i$ the set $M_i\ominus B.$ We note that $BX_iB\subset X_i$ and $X_iX_i \subset B \oplus X_i.$ Let $\underline{i}=(i_1,\ldots,i_k)$ and $\underline{j} = (j_1,\ldots,j_\ell)$ be possibly empty tuples of 1s and 2s such that $i_1\neq\ldots\neq i_k$ and $j_1\neq\ldots\neq j_\ell.$ We take $x\in X_{i_1}\cdots X_{i_k}$ and $y\in X_{j_1}\cdots X_{j_\ell}$ with the convention that if $\underline{i}$ is empty then $x\in B$, and similarly for $y.$ Since $BX_iB\subset X_i$ and $X_iX_i \subset B \oplus X_i,$ if $n$ is greater than the number of occurrences of the number 2 in $\underline{i}$ and $\underline{j}$ combined, we will have that $x(u_1u_2)^ny$ is a linear combination of words in $X_{m_1}\cdots X_{m_r}$ where at least one of the $m_s$ is equal to 2. Therefore $E_{M_1}(x(u_1u_2)^ny) = 0$ for all $n$ sufficiently large. By the easy direction of the intertwining theorem, the unitary $u_1u_2$ is not conjugate to any unitary in $M_1$. Since $M_1$ is diffuse, this means $M$ contains two non-conjugate unitaries.

    To analyze the case of an HNN extension, we use the notation of Section \ref{hnn-ext}. If $M = \mathrm{HNN}(N,A,\theta)$, then $M = \langle N,w\rangle''$ where $w$ is a Haar unitary and $\langle N,w\rangle'' \subset N_2 *_{A*A} B$ for some diffuse algebras $N_2,B$ such that $N\subset N_2.$ Furthermore, we can write $w = s^{-1}rt$ where $r,s,t$ are unitaries in $N_2 *_{A*A} B$ such that $N_2 \ni ts^{-1} \perp A*A$ and $B \ni r \perp A*A.$ By the previous paragraph, we have that $\|E_{N_2}(x(ts^{-1}r)^ny)\|_2 \to 0$ for all $x,y\in N_2 *_{A*A} B$. In particular, $\|E_{N_2}(xt(s^{-1}rt)^{n-1}s^{-1}ry)\|_2 \to 0$ for all $x,y\in N_2 *_{A*A} B$. Specializing to $N\subset N_2,$ we have that $\|E_N(xw^ny)\|_2\to 0 $ for all $x,y\in N.$ Since $N$ is diffuse, it contains a Haar which must not be conjugate to $w.$


    (3) If $M$ had a separable MASA $A,$ it would be isomorphic to $L\Z.$ The canonical group unitaries $u_1,u_2$ would be conjugate, but $\{u_2\}'' = L(2\Z)$ does not generate a MASA, a contradiction.
\end{proof}

\begin{remark}
    
   We also include the following auxillary elementary observation: let $\Gamma$ be a discrete group and if $g,h\in \Gamma$ are such that $g^n$ is not conjugate to $h^n$ for all nonzero integers $n,$ then $u_g$ is not conjugate to $u_h$ in $L\Gamma.$ Suppose towards a contradiction that $u_g$ and $u_h$ are conjugate in $L\Gamma$; then, there is a unitary $u\in L\Gamma$ such that $uu_g = u_h u.$ Write $u = \sum_{k\in\Gamma}\alpha_ku_k$. Then $\sum_{k}\alpha_ku_{kg} = \sum_k \alpha_k u_{hk}$. This implies that $\alpha_{h^{-1}k} = \alpha_{kg^{-1}}$ for all $k\in \Gamma.$ Equivalently, we have $\alpha_{h^{-n}kg^n} = \alpha_k$ for all integers $n$ and $k\in\Gamma$. But we know that $h^{-n}kg^n \neq k$ for any $n$; otherwise, $g^n$ and $h^n$ would be conjugate in $\Gamma.$ Therefore $\alpha_k = 0$, and so $u = 0,$ a contradiction.

\end{remark} 
\subsection{On conjugating Property (T) subalgebras}\label{sect-pf-of-(T)-conj}

We thank Adrian Ioana for suggesting we
 investigate conjugacy of property (T) subalgebras. We provide here a more or less optimal picture of the case of conjugation for property (T) subalgebras.

Let $(M,\tau)$ be a tracial von Neumann algebra. For a map $\phi :M \to M,$ define $$\|\phi\|_{\infty,2} = \sup\{\|\phi(x)\|_2 : x \in (M)_1\}.$$

We recall the notion of uniform distance between subalgebras of a tracial von Neumann algebra $(M,\tau)$ given by the metric $$d(A,B) = \max\{\|(I-E_B)E_A\|_{\infty,2},\|(I-E_A)E_B\|_{\infty,2}\}.$$ (See, e.g., \cite{PSSperturbations,WangPerturbations,christensen1979subalgebrasoffinite}). The metric $d$ is equivalent to the metric $d_2(A,B) = \|E_A - E_B\|_{\infty,2}$. More precisely, $d(A,B) \leq d_2(A,B) \leq \sqrt{2d(A,B)}$ (Remark 6.6 in \cite{PSSperturbations}).

It is known that the set of relatively rigid subalgebras of $M$ is closed in $d$ (Proposition 3.2 in \cite{WangPerturbations}). Here we prove the following, using the separability argument (for instance see Proof of Theorem 4.5.1 in \cite{PopaCorr}).

\begin{prop}\label{rigid-dense}
    Let $(M,\tau)$ be a separable tracial von Neumann algebra. Then the set of subalgebras of $M$ with property (T) is separable with respect to $d.$
\end{prop}

\begin{proof}
    Suppose for a contradiction that the set of subalgebras of $M$ with property (T) is non-separable. Then there is $\ee>0$ such that $M$ has property (T) subalgebras $(N_\alpha)_{\alpha\in I}$ such that $I$ is uncountable and for each $\alpha\neq\beta \in I,$ $d(N_\alpha,N_\beta) > \eps.$
    
    By property (T) of $N_\alpha\subset M$ for each $\alpha \in I$ and Lemma \ref{stupid-T-lemma}, there are finite sets $F_\alpha\subset N_\alpha$ and $\delta_\alpha > 0$ such that if $\phi:M\to M$ is a unital, tracial, completely positive map such that $\max_{x\in F_\alpha}\|\phi(x)-x\|_2 < \delta_\alpha$ then $\|\phi(y)-y\|_2 < \eps$ for all $y\in (N_\alpha)_1.$ 
    
    Since $I$ is uncountable, there are $m,n\in\N$ such that $I_0 = \{\alpha\in I : \delta_\alpha \geq \frac1n \text{ and } |F_\alpha|=m\}$ is uncountable. Since $M$ is separable, so is $L^2(M)^{\oplus m}$. Therefore there exist $\alpha\neq \beta \in I_0$ such that, writing $F_\alpha = (x_1,\ldots,x_m)$ and $F_\beta=(y_1,\ldots,y_m)$, we have $\|x_i-y_i\|_2 < \frac{1}{2n}$ for each $i=1,\ldots,m.$

    The map $E_{N_\alpha}:M\to M$ is unital, tracial, and completely positive and $$\|E_{N_\alpha}(y_i)-y_i\|_2 \leq \|E_{N_\alpha}(y_i-x_i)\|_2 + \|y_i-x_i\|_2 < \frac1n $$
    so by the relative rigidity of $N_\beta \subset N$ we get that $\|E_{N_\alpha}(x)-x\|_2 \leq \eps$ for all $x\in (N_\beta)_1$. Similarly, $\|E_{N_\beta}(x)-x\|_2 \leq \eps$ for all $x\in (N_\alpha)_1$. Hence $d(N_\alpha,N_\beta) \leq \eps,$ a contradiction. 

\end{proof}

The following is immediate from  Theorem 5.2 in\cite{PSSperturbations} and the inequality $d_2(A,B) \leq \sqrt{2d(A,B)}$ (see also Corollary 2.2 of \cite{WangPerturbations}). 

\begin{prop}\label{cor-of-pss}
    If $A,B\subset (M,\tau)$ are von Neumann subalgebras of a \twoone factor $M$ and $d= d(A,B) < \frac{1}{9522}$ then $A$ and $B$ are stably isomorphic. More specifically, there are projections $p\in A$ and $q\in B$ such that $pAp \cong qBq$ and $\tau(p),\tau(q) \geq 1-2450d.$
\end{prop}

Compared to the statement of Theorem 5.2 in \cite{PSSperturbations}, we note that we can drop the projections in $A'$ and $B'$ for the following reason: if $p' \in A',$ then $p'A$ is a SOT-closed two-sided ideal in $A$ and is therefore isomorphic to $zA$ for some $z\in Z(A).$ Similarly for $p'pAp$ for any projection $p\in A.$

Using HNN extensions allows us to conjugate not just isomorphic subalgebras, but ``almost'' conjugate ``almost'' isomorphic subalgebras. Recall that if $(M,\tau_M)$ is a tracial von Neumann algebra and $p\in M$ is a projection then $pMp$ inherits the canonical trace $\tau_{pMp}(pxp) = \frac{\tau_M(pxp)}{\tau_M(p)}$.

\begin{lem}\label{uauc-lem}
    If $A,B\subset (M,\tau_M)$ are von Neumann subalgebras of $M$ and $p\in A$, $q\in B$ are projections such that $pAp\cong qBq$ via a trace-preserving isomorphism, then there is a tracial von Neumann algebra $(N,\tau_N)$ such that $M\subset N,$ $\tau_N|_M = \tau_M,$ and there is a unitary $u\in N$ such that $d(uAu^*,B) \leq 5\sqrt{1-\min\{\tau(p),\tau(q)\}}$.
\end{lem}

\begin{proof}
    Without loss of generality assume that $\tau(p) \geq \tau(q).$ Set $M_1 = M \mathbin{\bar{\otimes}} L\Z$. Take a projection $p' \in L\Z$ such that $\tau(p') = \frac{\tau(q)}{\tau(p)}.$ Then $\Tilde{A} = p'pAp \oplus (1-p'p)\C$ and $\Tilde{B} = qBq \oplus (1-q)\C$ are isomorphic (via a trace-preserving isomorphism) subalgebras of $M_1.$ Let $\theta$ denote a trace-preserving embedding from $\Tilde{A}$ into $M_1$ with image $\Tilde{B}$. Now take $N = \mathrm{HNN}(M_1,\Tilde{A},\theta).$ There is a unitary $u\in N$ such that $uxu^* = \theta(x)$ for all $x\in \Tilde{A}$, and in particular $u\Tilde{A}u^* = \Tilde{B}.$

    Since $d(A,\Tilde{A}) = d(uAu^*, u\Tilde{A}u^*) = d(uAu^*,\Tilde{B})$, we have that $d(uAu^*,B) \leq d(A,\Tilde{A}) + d(\Tilde{B},B)$. But for $x\in (A)_1$, $\|x-p'pxp\|_2 \leq \|1-p'\|_2 + 2\|1-p\|_2 \leq 3\sqrt{1-\tau(p)}$. Thus $d(A,\Tilde{A}) \leq 3\sqrt{1-\tau(p)}$. Similarly, $d(B,\Tilde{B}) \leq 2\sqrt{1-\tau(q)}$.
\end{proof}

The following lemma is likely well-known to experts but we include a proof for completeness.
\begin{lem}\label{corner-of-corner-lem}
    If $(N,\tau)$ is a tracial von Neumann algebra with projections $p,q$ such that $\tau(p),\tau(q) > 1-\delta$ then there are projections $p'\leq p$ and $q'\leq q$ such that $\tau(p')=\tau(q') \geq 1-2\delta$ and $p'Np' \cong q'Nq'.$
\end{lem}

\begin{proof}
    By the Comparison Theorem for projections, there is a central projection $z\in Z(N)$ and partial isometries $v,w\in N$ such that $vv^* = pz,$ $v^*v \leq qz,$ $ww^* \leq p(1-z),$ and $w^*w = q(1-z)$. Take $p' = (v+w)(v+w)^*$ and $q'=(v+w)^*(v+w).$
\end{proof}

The following is due to Christensen; it can also be found in the Appendix of \cite{PopaCorr}.
\begin{thm}[Theorem 4.6 in \cite{christensen1979subalgebrasoffinite}]\label{christensen-thm}
    Let $A,B\subset (M,\tau)$ with $M$ type II$_1$, $A$ diffuse and $B$ a subfactor. Suppose $\|E_B(x)-x\|_2 \leq \delta$ for all $x\in (A)_1$. If $0<\delta<10^{-6}$ then there are projections $e\in A$ and $f\in B$ and a unital homomorphism $\Phi:eAe \to fBf$ such that $\|1-e\|_2 < 2\sqrt{\delta}$ and $\|\Phi(exe)-x\|_2 < 80\sqrt{\delta}$ for all $x\in (A)_1.$
\end{thm}

\begin{defn}
    Let $(M,\tau)$ be a tracial von Neumann algebra. We say that two von Neumann subalgebras $A,B\subset N$ are \emph{uniformly approximately unitarily equivalent (u.a.u.e.)} if for all $\eps>0$ there exists a unitary $u\in M$ such that $d(uAu^*,B) < \eps.$
\end{defn}

\begin{remark}
    If $N_1,N_2\subset N$ are u.a.u.e. type \twoone subfactors it does not follow that $N_1\cong N_2$; however, in light of Proposition \ref{cor-of-pss} it does imply  that there is a sequence $t_n\to 1$ such that $N_1^{t_n} \cong N_2.$ Moreover, these stable isomorphisms will be implemented by partial isometries in $N$ and $N_1,N_2$ will be mutually s-intertwining via surjective *-homomorphisms. (More precisely, for all projections $p_i \in N_i'\cap M$, we have $p_iN_i \prec_M N_{3-i}$ for $i=1,2,$ and the homomorphisms given by Theorem \ref{thm-popa-fundamental} can be taken to be surjective.)

    We recall that there are separable \twoone factors with prescribed countable fundamental group, see for instance \cite{Ho07,popa2010actions}. If the fundamental groups $\cF(N_1)$ and $\cF(N_2)$ are discrete subgroups of $\R_+$ then $N_1$ and $N_2$ being u.a.u.e will imply $N_1\cong N_2$. On the other hand, let $(N,\tau_N)$ be a \twoone factor with fundamental group $\Q^+$. Let $M$ be an isomorphic copy of $qNq$, where $\tau_N(q) = 2^{-1/2}$. Take positive numbers $r_n\in\Q\sqrt{2}$ which are increasing to $1$. Then there are projections $p_n\in N$ such that $\tau_N(p_n) = r_n$. Since the fundamental group of $N$ is $\Q^+,$ $p_nNp_n \cong M$ for all $n.$ By Lemma \ref{uauc-lem} we can inductively extend $N\overline{\otimes}M$ to $\tilde{N}$ to include unitaries $u_n$ such that $d(u_n(p_nNp_n)u_n^*,M) \to 0.$ But $d(p_nNp_n,N)\to 0$ too so that $d(u_nNu_n^*,M)\to 0.$ So we have constructed an algebra $\tilde{N}$ in which $N$ and $M$ are u.a.u.e. even though they are not isomorphic.
\end{remark}

We introduce one more definition before proving more general versions of Theorem \ref{main-(T)}. We call a family $(M_\lambda)_{\lambda\in\Lambda}$ of tracial von Neumann algebras an embedding universal inductive class of separable tracial von Neumann algebras if (1) it is closed under taking inductive limits and (2) for every separable tracial von Neumann algebra $N,$ there is an index $\lambda\in\Lambda$ such that $N$ embeds into $M_\lambda.$

\begin{thm}
    Let $(M,\tau)$ be a separable tracial von Neumann algebra. Then there is a separable tracial von Neumann algebra $N$ containing $M$ such that any two isomorphic diffuse property (T) von Neumann subalgebras of $N$ are u.a.u.e. We may furthermore take $N$ to belong to any embedding universal inductive class of separable tracial von Neumann algebras; e.g., we may take $N$ to be a \twoone factor or to have property Gamma. 
\end{thm}

\begin{proof}
    We take $M_1$ to equal $M$. By Proposition \ref{rigid-dense} enumerate a $d$-dense set  $(B_{m,1})_{m=1}^\infty$ of property (T) subalgebras of $M_1$. Fix a bijection $\sigma=(\sigma_1,\sigma_2,\sigma_3):\N\to\N^3$ such that $\sigma_3(n)\leq n.$ Define $\sigma_{13}(n) =(\sigma_1(n),\sigma_3(n))$ and similarly for $\sigma_{23}$.

    Suppose inductively that $M_1,\ldots,M_n$ have been constructed, along with $d$-dense sets of rigid subalgebras $(B_{m,k})\subset M_k$ for $m\in\N$ and $k=1,\ldots,n.$ If there exist projections $p_i\in B_{\sigma_{i3}(n)}$ such that $p_1B_{\sigma_{13}(n)}p_1 \cong p_2B_{\sigma_{23}(n)}p_2$ and $\tau(p_i) > 1 - \frac{1}{n}$ for $i=1,2$, using Proposition \ref{uauc-lem} define $M_{n+1}$ to be a \twoone factor containing $M_n$ and a unitary $u$ such that $d(uB_{\sigma_{13}(n)}u^*,B_{\sigma_{23}(n)})\leq 5\frac{1}{\sqrt{n}}$. Otherwise take $M_{n+1} = M_n.$

    Take $N$ to be the inductive limit of the $M_n.$ Fix $\eps>0.$ Let $N_1,N_2\subset N$ be two isomorphic diffuse property (T) subalgebras. By property (T) and Lemma \ref{stupid-T-lemma}, there exist $F_i\subset N_i$ finite subsets and $\delta>0$ such that if $\phi:N\to N$ is unital, tracial, and completely positive and $\|\phi(x)-x\|_2 < \delta$ for all $x\in F_i$ then $\|\phi(y)-y\|_2 < \eps$ for all $y\in (N_i)_1$. In particular, for all $n\in\N$ sufficiently large, $\|E_{M_n}(x)-x\|_2 < \delta$ for all $x\in F_i$ and so $\|E_{M_n}(y)-y\|_2 < \eps$ for all $y\in (N_i)_1,$ $i=1,2.$ We may in particular assume that $\frac{1}{\sqrt{n}} < 630\eps^{1/4}$. 

    By Theorem \ref{christensen-thm}, for $\eps$ sufficiently small, say $\eps < \frac{1}{2024!}$, and for $i=1,2$ there are projections $e_i\in N_i$ and $f_i\in M_n$ and *-homomorphisms $\Phi_i : e_iN_ie_i \to f_iM_nf_i$ such that $\|e_i-1\|_2 < 2\sqrt{\eps}$ and $\|\Phi_i(e_ixe_i)-x\|_2 \leq 80\sqrt{\eps}$ for all $x\in (N_i)_1.$ 

    Define $\Tilde{N_i} = \Phi_i(e_iN_ie_i)\oplus (1-f_i)\C$. Then $d(N_i,\Tilde{N_i}) \leq 80\sqrt{\eps}.$ By the $d$-density of the $B_{m,n}$ in the rigid subalgebras of $M_n$, we can find $n\leq k\in \N$ such that $\sigma_3(k) = n$ and for $i=1,2$, $d(\Tilde{N_i},B_{\sigma_{i3}(k)}) \leq \sqrt{\eps}$. Now Proposition \ref{cor-of-pss} implies that there are projections $p_i\in N_i$ and $q_i\in B_{\sigma_{i3}(k)}$ such that $p_iN_ip_i \cong q_iB_{\sigma_{i3}(k)}q_i$ and $\tau(p_i),\tau(q_i) \geq 1-198450\sqrt{\eps}$. By Lemma \ref{corner-of-corner-lem}, there are projections $s_i \leq p_i$ in $N_i$ such that $s_1N_1s_1$ is isomorphic to $s_2N_2s_2$ and $\tau(s_i) \geq 1-396900\sqrt{\eps}$ for $i=1,2$. In turn, this implies that there are projections $r_i\leq q_i$ in $B_{\sigma_{i3}(k)}$ such that $r_iB_{\sigma_{i3}(k)}r_i$ are isomorphic to $r_i$ for $i=1,2$ and $\tau(r_i) \geq 1-396900\sqrt{\eps}$. By construction, there is a unitary $u\in M_{k+1}\subset N$ such that $d(uB_{\sigma_{13}(k)}u^*,B_{\sigma_{23}(k)}) \leq \frac{5}{\sqrt{k}}\leq \frac{5}{\sqrt{n}} < 3150\eps^{1/4}$.

    Applying the triangle inequality repeatedly, we get that $d(uN_1u^*,N_2) \leq 3312\eps^{1/4}$. Since $\eps$ was chosen arbitrarily, we see that $N_1$ and $N_2$ are u.a.u.e.
\end{proof}

The above proof actually shows the following stronger statement, which moreover provides a case for why the conclusion is optimal (see also the Remark \ref{remark crazy}). 

\begin{thm}\label{stronger statement}
    Let $(M,\tau)$ be a separable tracial von Neumann algebra. Then there is a separable \twoone factor $N$ containing $M$ such that any two diffuse isomorphic property (T) von Neumann subalgebras $N_1$ and $N_2$ of $N$ satisfy that there exists a sequence of projections $p_{n,i}\to 1$ (where $p_{n,i}\in \mathcal{P}(N_i)$) such that $p_{n,1}N_1p_{n,1}\cong p_{n,2}N_2p_{n,2}$, are u.a.u.e. 
\end{thm}

\begin{remark}\label{remark crazy}
    We note that if $M$ is a \twoone factor with property (T) then $M\mathbin{\bar{\otimes}} R$ has full fundamental group, and therefore $M^t \subset (M\mathbin{\bar{\otimes}} R)^t\cong M\mathbin{\bar{\otimes}}R$ for all $t\in \R_+$. In particular, $M\mathbin{\bar{\otimes}} R$ contains continuum many non-isomorphic, stably isomorphic property (T) subfactors since the fundamental group of $M$ is countable. We contrast this remark with the fact that there are only countably many irreducible subfactors with (T) (see \cite{PopaCorr}). This suggests that it is not likely possible to be able to exactly conjugate all isomorphic property (T) subalgebras inside of a separable \twoone factor, since our inductive limit arguments would require passing to the non-separable setting.
\end{remark}

\subsection*{Miscellaneous results} 


The following is well known and is a folklore fact. 
\begin{lem}\label{axiom-factor}
    Let $(M, \tau_M)$ be a tracial von Neumann algebra. Then $M$ is a factor if and only if for every $x \in M$ and $\eps > 0$, there exists $y \in (M)_1$ such that $\|x - \tau(x)\|_2 < \|[x, y]\|_2 + \eps$.
\end{lem}

The proof of the following  proposition was suggested to us by D. Jekel and B. Hayes. We thank them profusely for allowing us to include it here. 

\begin{prop}\label{factor-lifting}
    Let $\{x_1, \cdots, x_n\} \subset (M)_1$ be a finite set of elements of operator norm at most $1$ in a \twoone factor $M$. Then for any $\eps > 0$, there exists $\{y_1, \cdots, y_n\} \subset (M)_1$, elements of operator norm at most $1$, such that $\{y_1, \cdots, y_n\}$ generates a subfactor of $M$ and furthermore $\|y_i - x_i\|_2 < \eps$ for all $i$.
\end{prop}

\begin{proof} Let $A = C^\ast(T_1, \cdots, T_n)$ be the universal $C^\ast$-algebra generated by $n$ contractions $T_1, \cdots, T_n$. For each positive integer $N > 0$ and for each $\ast$-polynomial in $n$ variables with rational coefficients $p$, we consider the following subset of $(M)_1^n$,
\begin{equation*}
\begin{split}
    G_{N, p} = \{(z_1, &\cdots, z_n) \in (M)_1^n: \, \exists \, \ast\textrm{-polynomial in}\, n \, \textrm{variables} \, q \, \textrm{such that} \, \|q\|_A \leq 1\\
    &\textrm{and} \, \|p(z_1, \cdots, z_n) - \tau(p(z_1, \cdots, z_n))\|_2 < \|[p(z_1, \cdots, z_n), q(z_1, \cdots, z_n)]\|_2 + \frac{1}{N}\}
\end{split}
\end{equation*}

We observe that the above set is open in $(M)_1^n$ in the 2-norm topology. Indeed, for each fixed $\ast$-polynomial $q$ in $n$ variables with $\|q\|_A \leq 1$, the set
\begin{equation*}
\begin{split}
    G_{N, p, q} = \{(z_1, \cdots&, z_n) \in (M)_1^n:\\
    &\|p(z_1, \cdots, z_n) - \tau(p(z_1, \cdots, z_n))\|_2 < \|[p(z_1, \cdots, z_n), q(z_1, \cdots, z_n)]\|_2 + \frac{1}{N}\}
\end{split}
\end{equation*}
is clearly open, and $G_{N, p}$ is the union of $G_{N, p, q}$ ranging over all such $q$. We now claim that $G_{N, p}$ is 2-norm dense in $(M)_1^n$ for each $N$ and $p$. Indeed, it suffices to show, for any $(v_1, \cdots, v_n) \in M^n$ with the operator norms of all $v_i$ strictly smaller than $1$, and any $\eps > 0$, there exists $(w_1, \cdots, w_n) \in G_{N, p}$ with $\|w_i - v_i\|_2 \leq \eps$ for all $i$. Let $\cU$ be a free ultrafilter on $\N$. As $v_1, \cdots v_n \in M \subset M^\cU$, by \cite{Popaindep} there exists a family of free semicirculars $s_1, \cdots, s_n \in M^\cU$ free from $v_1, \cdots v_n$. By \cite{dabrowskifullff}, for $\eps$ sufficiently small, $Q = W^\ast(v_1 + \eps s_1, \cdots, v_n + \eps s_n)$ is a factor. Furthermore, again by taking $\eps$ sufficiently small, we may assume $\|v_i\| + \eps \leq 1$. Write $z_i = v_i + \eps s_i$.

Since $Q$ is a factor, by the Kaplansky density theorem, Proposition II.5.1.5 in \cite{Bl06}, and Lemma \ref{axiom-factor}, there exists a $\ast$-polynomial $q$ in $n$ variables such that $$\|p(z_1, \cdots, z_n) - \tau(p(z_1, \cdots, z_n))\|_2 < \|[p(z_1, \cdots, z_n), q(z_1, \cdots, z_n)]\|_2 + \frac{1}{N}$$ and furthermore $\|q(z_1, \cdots, z_n)\|_\infty\leq\|q\|_A \leq 1$ (see also the argument in the proof of Lemma 2.2 in \cite{JP23} and Lemma 2.3 in \cite{Hayes2018}). Now, lift each $s_i$ to a sequence of elements $(s_{ij})_{j \to \cU}$ in $M$ such that $\|s_{ij}\|_\infty \leq 1$ for all $i$ and $j$. Write $z_{ij} = v_i + \eps s_{ij}$. Then $\|z_{ij}\|_\infty \leq \|v_i\|_\infty + \eps \leq 1$ and $\|z_{ij} - v_i\|_2 \leq \eps$. Furthermore, $(z_{ij})_{j \to \cU}$ is a lift of $z_i$. Thus,
\begin{equation*}
\begin{split}
    \lim_{j \to \cU} (&\|p(z_{1j}, \cdots, z_{nj}) - \tau(p(z_{1j}, \cdots, z_{nj}))\|_2 - \|[p(z_{1j}, \cdots, z_{nj}), q(z_{1j}, \cdots, z_{nj})]\|_2)\\
    = &\|p(z_1, \cdots, z_n) - \tau(p(z_1, \cdots, z_n))\|_2 - \|[p(z_1, \cdots, z_n), q(z_1, \cdots, z_n)]\|_2\\
    < &\frac{1}{N}
\end{split}
\end{equation*}

So there exists $S \in \cU$ such that whenever $j \in S$, we have $\|p(z_{1j}, \cdots, z_{nj}) - \tau(p(z_{1j}, \cdots, z_{nj}))\|_2 < \|[p(z_{1j}, \cdots, z_{nj}), q(z_{1j}, \cdots, z_{nj})]\|_2 + \frac{1}{N}$. Setting $w_i = z_{ij}$ for any $j \in S$ proves the claim.

Since we are restricting to $\ast$-polynomials with rational coefficients $p$, there are only countably many such $p$. Hence, by Baire category theorem, the set
\begin{equation*}
    G = \bigcap_{\begin{matrix}N \in \N_+ \\
    p \, \ast\textrm{-polynomial with rational coefficients}\end{matrix}} G_{N, p}
\end{equation*}
is dense in $(M)_1^n$. In particular, there exists $(y_1, \cdots, y_n) \in G$ such that $\|x_i - y_i\|_2 < \eps$ for all $i$. We claim that $P = W^\ast(y_1, \cdots, y_n)$ is a factor. Indeed, for any $x \in P$ and $\eps > 0$, there exists a $\ast$-polynomial $p$ in $n$ variables with rational coefficients such that $\|x - p(y_1, \cdots, y_n)\|_2 < \frac{\eps}{5}$. Let $N > 0$ be a positive integer such that $\frac{1}{N} \leq \frac{\eps}{5}$. Then as $(y_1, \cdots, y_n) \in G \subset G_{N, p}$, there exists a $\ast$-polynomial $q$ in $n$ variables with $\|q\|_A \leq 1$ such that
\begin{equation*}
\begin{split}
    \|x - \tau(x)\|_2 &\leq \|p(y_1, \cdots, y_n) - \tau(p(y_1, \cdots, y_n))\|_2 + \frac{2\eps}{5}\\
    &< \|[p(y_1, \cdots, y_n), q(y_1, \cdots, y_n)]\|_2 + \frac{1}{N} + \frac{2\eps}{5}\\
    &\leq \|[x, q(y_1, \cdots, y_n)]\|_2 + \frac{2\|q(y_1, \cdots, y_n)\|_\infty\eps}{5} + \frac{3\eps}{5}
\end{split}
\end{equation*}

Since $\|y_i\|_\infty \leq 1$ for all $i$, as $\|q\|_A \leq 1$, we have $\|q(y_1, \cdots, y_n)\|_\infty \leq 1$, so the above yields that, for any $x \in P$ and $\eps > 0$, there exists $y \in (M)_1$, namely $y = q(y_1, \cdots, y_n)$, such that $\|x - \tau(x)\|_2 < \|[x, q(y_1, \cdots, y_n)]\|_2 + \eps$. Thus, $P$ is a factor by Lemma \ref{axiom-factor}.
\end{proof}

\begin{defn}
Let $(N, \tau_N)$, $(M, \tau_M)$ be separable \twoone factors, $\{x_i\}_{i \in I} \subset N$ be a countable set of elements of operator norm at most $1$ that generates $N$, and $\{y_j\}_{j=1}^n \subset M$ be a finite set of elements of operator norm at most $1$ that generates $M$. Let $I_0 \subset I$ be a finite subset, $\eps > 0$, and $m > 0$ be a positive integer. Then we say $(M, \tau_M, \{y_j\})$ is in the $(I_0, \eps, m)$\textit{-neighborhood} of $(N, \tau_N, \{x_i\})$, denoted by $(M, \{y_j\}) \in \Theta(N, I_0, \eps, m)$, if there exists an injective map $\sigma: \{1, \cdots, n\} \to I$ whose range contains $I_0$, such that,
\begin{equation*}
    |\tau_M(p(y_1, \cdots, y_n)) - \tau_N(p(x_{\sigma(1)}, \cdots, x_{\sigma(n)}))| < \eps
\end{equation*}
for all $\ast$-monomials $p$ in $n$ variables of degree less than or equal to $m$. If we need to specify the correspondence between generators, i.e., the injective map $\sigma$, we shall write $(M, \{y_j\}, \sigma) \in \Theta(N, I_0, \eps, m)$. 

If, for $(M, \tau_M)$, there exists $\{y_j\}_{j=1}^n \subset M$, a finite set of elements of operator norm at most $1$ that generates $M$, such that $(M, \{y_j\}) \in \Theta(N, I_0, \eps, n)$, we then say $(M, \tau_M)$ is in the $(I_0, \eps, m)$\textit{-neighborhood} of $(N, \tau_N, \{x_i\})$, denoted by $M \in \Theta(N, I_0, \eps, n)$.

If $I$ is finite and $I_0 = I$, we simply write $\Theta(N, \eps, m)$ in place of $\Theta(N, I_0, \eps, m)$.
\end{defn}

\begin{prop}\label{gamma-prop}
    Let $N$, $M$ be \twoone factors, $\{x_i\}_{i \in I} \subset N$ be a countable set of elements of operator norm at most $1$ that generates $N$, $\mathcal{U}$ be a countably incomplete ultrafilter on an index set $J$, then the following are equivalent:
    \begin{enumerate}
        \item For any countably incomplete ultrafilter $\mathcal{V}$ on any index set, for any embedding $\pi: N \to M^\mathcal{V}$, the relative commutant $\pi(N)' \cap M^\mathcal{V}$ contains a trace-zero unitary;
        \item For any embedding $\pi: N \to M^\mathcal{U}$, the relative commutant $\pi(N)' \cap M^\mathcal{U}$ contains a trace-zero unitary;
        \item For any $\eps > 0$ and finite subset $I_0 \subset I$, there exists a finite subset $I' \subset I$ containing $I_0$, $\delta > 0$, and a positive integer $n > 0$ such that whenever $(N_0, \{y_j\}, \sigma) \in \Theta(N, I', \delta, n)$ and $\pi: N_0 \to M$ is an embedding, then there exists a unitary $w \in M$ of trace zero such that $\|[y_{\sigma^{-1}(i)}, w]\|_2 \leq \eps$ for all $i \in I_0$.
    \end{enumerate}
\end{prop}

\begin{proof}
    $(1) \Rightarrow (2)$ is clear.

    $(2) \Rightarrow (3)$ Assume to the contrary; i.e., assume there exists $\eps > 0$ and a finite subset $I_0 \subset I$ such that for any finite subset $I' \subset I$ containing $I_0$, $\delta > 0$, and positive integer $n > 0$, there exists $(N_0, \{y_j\}, \sigma) \in \Theta(N, I_0, \delta, n)$ and $\pi: N_0 \to M$ an embedding such that for any unitary $w \in M$ of trace zero, we have $\|[\pi(y_{\sigma^{-1}(i)}), w]\|_2 > \eps$ for some $i \in I_0$.
    
    Thus, fix an increasing sequence of finite subsets $(I_k)_k \subset I$ containing $I_0$ and whose union is $I$, and a decreasing sequence $\delta_k \to 0$. Then for each $k$, there exists $(N_k, \{y_j\}_{j=1}^{m_k}, \sigma_k) \in \Theta(N, I_k, \delta_k, k)$ and $\pi_k: N_k \to M$ a trace-preserving embedding such that for any unitary $w \in M$ of trace zero, we have $\|[\pi_k(y_{\sigma_k^{-1}(i)}), w]\|_2 > \eps$ for some $i \in I_0$. As $\mathcal{U}$ is countably incomplete, we may choose a decreasing sequence of sets $(J_k)_k \in \mathcal{U}$ whose intersection is empty. We then define, for each $j \in J$, a map $\varphi_j: \{x_i\}_{i \in I} \to M$,
    \begin{equation*}
        \varphi_j(x_i) = \begin{cases}
            \pi_k(y_{\sigma_k^{-1}(i)}) &, \, \mathrm{if} \, j \in J_k \setminus J_{k+1} \, \mathrm{and} \, i \in I_k\\
            0 &, \, \mathrm{otherwise}
        \end{cases}
    \end{equation*}
    
    Then the map $\varphi: \{x_i\}_{i \in I} \to M^\mathcal{U}$ given by $\varphi(x_i) = (\varphi_j(x_i))_{j \to \mathcal{U}}$ is easily seen to preserve the law of $\{x_i\}_{i \in I} \subset N$ under $\tau_N$. Since $\{x_i\}$ generates $N$, $\varphi$ extends to an embedding $N \to M^\mathcal{U}$, which we shall still denote by $\varphi$.

    By our assumptions, we may pick $w \in \varphi(N)' \cap M^\mathcal{U}$ a trace-zero unitary. Lift it to $w = (w_j)_{j \to \mathcal{U}}$ where $w_j \in M$ are all trace-zero unitaries, so $\|[\varphi_j(x_i), w_j]\|_2 \to 0$ as $j \to \mathcal{U}$ for any fixed $i \in I$. In particular, as $I_0 \subset I$ is finite, there exists $A \in \mathcal{U}$ such that $\|[\varphi_j(x_i), w_j]\|_2 \leq \eps$ whenever $i \in I_0$ and $j \in A$. Since $A \cap J_1 \in \mathcal{U}$, $A \cap J_1 \neq \varnothing$. Since $J_1 = \cup_{k = 1}^\infty (J_k \setminus J_{k+1})$, $A \cap (J_k \setminus J_{k+1}) \neq \varnothing$ for some $k$. Fix $j \in A \cap (J_k \setminus J_{k+1})$. Then for any $i \in I_0 \subset I_k$, we have,
    \begin{equation*}
        \|[\pi_k(y_{\sigma_k^{-1}(i)}), w_j]\|_2 = \|[\varphi_j(x_i), w_j]\|_2 \leq \eps
    \end{equation*}

    But by assumptions on $N_k$ and $\pi_k$, we have $\|[\pi_k(y_{\sigma_k^{-1}(i)}), w_j]\|_2 > \eps$ for some $i \in I_0$, a contradiction.

    $(3) \Rightarrow (1)$ Let $\pi: N \to M^\mathcal{V}$ be an embedding. We may write $\pi(x_i) = (x_{ji})_{j \to \mathcal{V}}$ where $x_{ji}$ are elements of $M$ of operator norm at most $1$. Fix an increasing sequence of finite subsets $I_k \subset I$ whose union is $I$. Then by assumptions there exists finite subsets $I'_k \subset I$ containing $I_k$, $\delta_k > 0$, and a positive integer $n_k > 0$ such that whenever $(N_0, \{y_j\}, \sigma) \in \Theta(N, I'_k, \delta_k, n_k)$ and $\pi: N_0 \to M$ is a trace-preserving embedding, then there exists a unitary $w \in M$ of trace zero such that $\|[\pi(y_{\sigma^{-1}(i)}), w]\|_2 \leq \frac{1}{k}$ for all $i \in I_k$.

    Note that since the laws for $\{x_{ji}\}_{i \in I} \subset M$ converge to the law of $\{x_i\}_{i \in I}$ as $j \to \mathcal{V}$, for each $k$ there exists $A_k \in \cV$ such that, whenever $j \in A_k$, if we write $I'_k = \{i_{k, 1}, \cdots, i_{k, |I'_k|}\}$, then,
    \begin{equation*}
        |\tau_M(p(x_{ji_{k, 1}}, \cdots, x_{ji_{k, |I'_k|}})) - \tau_N(p(x_{i_{k, 1}}, \cdots, x_{i_{k, |I'_k|}}))| < \frac{\delta_k}{2}
    \end{equation*}
    for all $\ast$-monomials $p$ in $|I'_k|$ variables of degree less than or equal to $n_k$. By Proposition \ref{factor-lifting}, for each $j \in A_k$, there exist $\{y_{jki}\}_{i \in I'_k} \subset M$ of operator norm at most $1$ such that $\{y_{jki}\}_{i \in I'_k}$ generates a subfactor of $M$, and furthermore,
    \begin{equation*}
        \|y_{jki} - x_{ji}\|_2 < \min \left\{ \frac{\delta_k}{2n_k}, \frac{1}{k} \right\}
    \end{equation*}
    for all $i \in I'_k$. But then it is clear that,
    \begin{equation*}
        |\tau_M(p(y_{jki_{k, 1}}, \cdots, y_{jki_{k, |I'_k|}})) - \tau_N(p(x_{i_{k, 1}}, \cdots, x_{i_{k, |I'_k|}}))| < \delta_k
    \end{equation*}
    for all $\ast$-monomials $p$ in $|I'_k|$ variables of degree less than or equal to $n_k$. That is, if $N_{jk} = W^\ast(\{y_{jki}\}_{i \in I'_k}) \subset M$, then $N_{jk}$ is a separable \twoone factor and $(N_{jk}, \{y_{jki}\}_{i \in I'_k}, \sigma_k) \in \Theta(N, I'_k, \delta_k, n_k)$, where $\sigma_k$ is simply the inclusion map $I'_k \hookrightarrow I$. By assumption, there exists a trace-zero unitary $w_{jk} \in M$ such that $\|[y_{jki}, w_{jk}]\|_2 \leq \frac{1}{k}$ for all $i \in I_k$. But then, as $\|y_{jki} - x_{ji}\|_2 < \frac{1}{k}$ for all $i \in I'_k \supset I_k$, we have $\|[x_{ji}, w_{jk}]\|_2 \leq \frac{3}{k}$ for all $i \in I_k$.

    As $\cV$ is countably incomplete, we may choose a decreasing sequence of sets $(J_k)_k \in \cV$ whose intersection is empty. Let $B_k = J_k \cap (\bigcap_{l = 1}^k A_k)$. Then $B_k$ is a sequence of sets in $\cV$ that decreases to $\varnothing$. We define, for each $j$,
    \begin{equation*}
        w_j = \begin{cases}
            w_{jk} &, \, \mathrm{if} \, j \in B_k \setminus B_{k+1}\\
            0 &, \, \mathrm{otherwise}
        \end{cases}
    \end{equation*}

    One easily checks that $w = (w_j)_{j \to \mathcal{U}}$ is a trace-zero unitary that commutes with $\pi(x_i) = (x_{ji})_{j \to \cV}$ for any $i \in I$; i.e., $\pi(N)' \cap M^\mathcal{V}$ contains a trace-zero unitary.
\end{proof}

We shall denote the equivalent conditions contained in the above proposition $\Gamma_N(M)$; i.e., we shall write $\Gamma_N(M)$ when, for any countably incomplete ultrafilter $\mathcal{V}$ on any index set and for any embedding $\pi: N \to M^\mathcal{V}$, the relative commutant $\pi(N)' \cap M^\mathcal{V}$ contains a trace-zero unitary. By condition (2) in the proposition above, we see that $\Gamma_N(P)$ is preserved by changing $P$ to $Q$ which is elementarily equivalent to $P$ when $P$ and $Q$ both have density characters at most continuum, i.e.,

\begin{cor}\label{ele-equiv-gamma}
    If $P$ and $Q$ are elementarily equivalent \twoone factors which have density characters at most continuum and $N$ is any fixed separable \twoone factor, then $\Gamma_N(P)$ if and only if $\Gamma_N(Q)$.
\end{cor}

One particular observation we make here is that, if $P$ is non-Gamma with density character at most continuum, and $Q \equiv P$ with $Q$ separable, then $Q$ is non-Gamma as property Gamma is preserved by elementary equivalence. Then, because $P^\cU \cong Q^\cU$ for some countably incomplete $\cU$ and the diagonal embedding $\Delta_Q: Q \to Q^\cU \cong P^\cU$ has trivial relative commutant, we have,

\begin{cor}\label{non-gamma-wrt-itself}
    If $P$ and $Q$ are elementarily equivalent non-Gamma \twoone factors with $Q$ separable and $P$ having density character at most continuum, then $\Gamma_{Q}(P)$ fails.
\end{cor}

\bibliographystyle{amsalpha}
\bibliography{inneramen}

\end{document}